\documentclass[10pt,twoside]{article}
\usepackage{amsfonts}
\usepackage{fancyhdr}
\usepackage{titlesec}
\usepackage{ifthen}
\usepackage{amsmath}

\titleformat{\section}{\centering\large\bfseries}{\S\arabic{section}}{1em}{}
\newboolean{first}
\setboolean{first}{true}

 \def\be{\begin{eqnarray}}
\def\ee{\end{eqnarray}}

\newcommand{\hs}{\hskip 0.3cm}

\renewcommand{\theequation}{\arabic{equation}}
\newcommand{\D}{\displaystyle}

\newcommand{\myref}[1]{(\ref{#1})}
\newcommand{\mathd}{\mathrm{d}}

\newenvironment{proof}{\noindent\textbf{Proof\ }}{\hspace*{\fill}$\Box$\medskip}

\newtheorem{theorem}{\textbf{Theorem}}[section]
\newtheorem{prop}[theorem]{\textbf{Proposition}}
\newtheorem{lemma}[theorem]{\textbf{Lemma}}

\newtheorem{rem}[theorem]{\textbf{Remark}}

\renewcommand{\theequation}{\thesection.\arabic{equation}}
\catcode`@=11 \@addtoreset{equation}{section} \catcode`@=12

\begin{document}
\setlength\abovedisplayskip{2pt}
\setlength\abovedisplayshortskip{0pt}
\setlength\belowdisplayskip{2pt}
\setlength\belowdisplayshortskip{0pt}
\title{\bf \Large  On one multidimensional compressible nonlocal model of the dissipative QG equations
\author{Shu Wang, Linrui Li and Shengtao Chen\\
College of Applied Sciences, Beijing University
of Technology,\\ PingLeYuan100, Chaoyang District, Beijing100124, P.
R. China\\
{E-mail:} wangshu@bjut.edu.cn;
lilinrui@emails.bjut.edu.cn;\\ chensht08@emails.bjut.edu.cn}\date{\,\,}
}
\maketitle
% \footnote{{2000 Mathematics Subject
%Classification: 35R10; 35A05; 35A07; 35L60.}   }
%\footnote{ }
\begin{minipage}{135mm} %%%%%%%%%%%%%%%%%%%%%%%%%%%%%%%%%%%%%%%%%%%%%%%%%%%%摘要部分
{\bf \small Abstract.}\hskip 2mm {\small In this paper we study the
Cauchy problem for one multidimensional compressible nonlocal model
of the dissipative quasi-geostrophic equations. First, we obtain the
local existence and uniqueness of the smooth non-negative solution
or the strong solution in time. Secondly, for the sub-critical and
critical case $1\le\alpha\leq 2$, we obtain the global existence and
uniqueness results of the nonnegative smooth solution. Then, we
prove the global existence of the weak solution for $0\le \alpha\le
2$ and $\nu\ge 0$. Finally, for the sub-critical case, we establish
the global $H^1$ and $L^p, p>2,$ decay rate of the smooth solution
as $t\to\infty$.}

\end{minipage}

\thispagestyle{fancyplain} \fancyhead{}
\fancyhead[L]{\textit{}\\
 } \fancyfoot{} \vskip 10mm
\setcounter{section}{1} \setcounter{equation}{0} %%%%%%%%%%%%%%%%%%%%%%%%%%%%%%%%%%%%%%%%%第一部分
\renewcommand{\theequation}
{\thesection.\arabic{equation}}
\begin{center}
{\large\bf\S\bf1\hspace*{3mm} Introduction and main results}
\end{center}
%%%%%%%%%%%%%%%%%%%%%%%%%%%%%%%%%%%%%%%%%%%%%%%%%%%%%%%%%%%%%%%%%%%%%%%%%%%%%%%%%%%%
In this paper, we study the following Cauchy problem for one
multidimensional compressible nonlocal model of the dissipative
quasi-geostrophic equations:
\begin{equation}\label{1.1}
\left\{
\begin{array}{lll}\partial_t\theta+div(u \theta)+\nu(-\Delta)^\frac{\alpha}{2}\theta=0,\hs x\in \mathbb{R}^N, t>0,\\
u=\mathcal {R}\theta=(\mathcal {R}_1 \theta, \mathcal {R}_2\theta, \cdots, \mathcal {R}_N\theta),\\
\theta(x, 0)=\theta_0(x).\end{array}\right.
\end{equation}
where $\theta: \mathbb{R}^N\rightarrow \mathbb{R}$ is a scalar
function of $x$ and $t$, representing potential temperature, $u(x,
t)$ is the velocity field of fluid given by the Riesz transform
$u=\mathcal {R}\theta=(\mathcal {R}_1\theta,\mathcal{R}_2\theta,
\cdots, \mathcal{R}_N\theta)$, defined by $$\mathcal
{R}_i(\theta)(x, t)=\frac{1}{(2\pi)^N}P.
V.\int_{\mathbb{R}^N}\frac{(x_i-y_i)\theta(y, t)}{|x-y|^{N+1}}dy,
i=1, 2,\cdots, N.$$ $\nu>0$ is the dissipative coefficient and
$N\geq 1$. We denote
$\Lambda^\alpha=(-\Delta)^{\frac{\alpha}{2}}$, which is defined by
the Fourier transform $F(\Lambda^\alpha)=F((-\Delta)^{\frac
\alpha2})=|\xi|^\alpha$. Moreover, for $0<\alpha\le 2$,
$\Lambda^\alpha\theta$ is given (see e.g. \cite{14,ju}) by
\be\Lambda^\alpha\theta(x)=C_\alpha
P.V.\int_{\mathbb{R}^N}\frac{\theta(x)-\theta(y)}{|x-y|^{N+\alpha}}dy,
x\in \mathbb{R}^N\label{reiz}\ee and, especially, $\Lambda\theta$
is given by \be\Lambda\theta(x)=C(N)
P.V.\int_{\mathbb{R}^N}\frac{\theta(x)-\theta(y)}{|x-y|^{N+1}}dy=div
\mathcal{R}\theta(x), x\in \mathbb{R}^N.\label{reiz1}\ee One
particular feature of system (\ref{1.1}) is the relation with the
dissipative quasi-geostrophic equations \cite{CMT}, which is
easily derived by changing the incompressible velocity field
$u=(-\mathcal {R}_2\theta,\mathcal {R}_1\theta)$ of surface QG
equations into the compressible velocity field $u=(\mathcal
{R}_1\theta,\mathcal {R}_2\theta)$ for $N=2$.  The system
\eqref{1.1} with $N=1$ and $\nu\ge 0$ was displayed by Baker et al
in \cite{BLM} as a one-dimensional model of the 2D Vortex sheet
problem, and was further investigated by D. Chae, et al
\cite{CCCF} and Castro and C\'{o}rdoba \cite{CC} and global
existence, finite time singularities and ill-posedness was
discussed therefore by using the theory of complex-value partial
differential equations. However, the methods used by D. Chae, et
al \cite{CCCF} and Castro and C\'{o}rdoba \cite{CC} can not be
applied to the present multidimensional problem. Hence, the main
purpose of this paper is that we extend the results for the ND
model \eqref{1.1} with $N=1$ given by D. Chae, et al \cite{CCCF}
and Castro and C\'{o}rdoba \cite{CC} to the general ND model
\eqref{1.1} for $N\ge 1$ by using completely different methods.

It should be pointed out that some multidimensional models related
to the dissipative or inviscid quasi-geostrophic equations have been
studied by many authors. P. Balodis and A. C\'ordoba \cite{BC2007}
discuss the blow-up problem for a class of nonlinear and nonlocal
transport equations by using an inequality for Riesz transforms. A.
Castro, D. Cordoba et al \cite{castro} study heat transfer with a
general fractional diffusion term of incompressible fluid in a
porous medium governed by Darcy's law and obtain local and global
wellposedness for the strong or weak solutions and the existence of
the global attractor for the solutions. The global existence and
finite time blow-up problems for the aggregation equations are
studied in \cite{thomas, 10, li} for some different singular
potentials. Recently, a porous medium equation with nonlocal
diffusion effects given by an inverse fractional Laplacian operator,
i.e,
\begin{equation}\partial_tu=\nabla\cdot(u\nabla p), p=(-\Delta)^{-s}u,
0<s<1, \label{por}\end{equation} is studied by L. Caffarelli and L.
Vazquez \cite{CV}. The existence of the global weak solution for the
nonnegative and bounded initial data function with compact support
or fast decay at infinity is proven. The existence and uniqueness of
the local or global smooth solution remain open. Notice that the
model \eqref{por} with $s=\frac 12$ have a different sign from the
ND model \eqref{1.1}. We will show that the local or global
existence of smooth solutions in time depends heavily upon the sign
of the solutions or the sign of initial data. In particular, if the
initial data $u(t=0)\le 0$ is a smooth function, then the model
\eqref{por} with $s=\frac 12$ and initial data $u(t=0)$ has a smooth
solution locally in time.

In this paper, we will investigate the general multidimensional
compressible nonlocal flux \eqref{1.1} for the case
$0\leq\alpha\leq2$, $\nu\ge 0$ and for the nonnegative initial data.
Here the case $\alpha=1$ is called the critical case, the case
$1<\alpha\leq2$ is so-called sub-critical one and the case
$0\leq\alpha<1$ is super-critical one. Roughly speaking, the
critical and super-critical cases are mathematically harder to deal
with than the sub-critical case.

We now state our main results. First, we give the following local
existence result for the smooth solution to the system \eqref{1.1}
with $\nu\ge 0$.
\begin{theorem}\label{thm1.1}
Let $ 0\le\alpha\le 2$ and $\nu\geq 0$. Assume that the initial data
$\theta_0\geq 0$ and $\theta_0\in H^s(\mathbb{R}^N)$ for some
positive integer $s>\frac{N}{2}+1$. Then (\ref{1.1}) has a unique
smooth solution $\theta\in C([0, T^*); H^s(\mathbb{R}^N))\bigcap
C^1([0, T^*); H^{s-2}(\mathbb{R}^N))$, defined on $[0, T^*)$,
$T^*=T(\|\theta_0\|_{H^s(\mathbb{R}^N)})>0$ is the maximal existence
time. Moreover, if $T^*<\infty$, then $$\int_0^T\|\theta(\cdot,
t)\|_{H^s}dt=\infty$$ or
$$\int_0^T(\|\Lambda\theta\|_{L^\infty}+\|\nabla\theta\|_{L^\infty}+\|\nabla
u\|_{L^\infty})dt=\infty.$$
\end{theorem}
\begin{rem} The assumption
that $\theta_0\ge 0$ in Theorem \ref{thm1.1} plays a key role in
obtaining local existence results on the smooth solution to the
system \eqref{1.1}. If we remove this assumption, in particular, we
assume that the initial data $\theta_0$ changes its sign in
$\mathbb{R}^N$, we can not get the local existence of the smooth
solution to the system \eqref{1.1} by using the method used in the
proof of Theorem \ref{thm1.1}. This will be discussed in the future.
\end{rem}
If $\nu>0$ and $0<\alpha\le 2$, then we can obtain the local
existence results on the strong solution to the system \eqref{1.1}.
\begin{theorem}\label{thm1.2}
Let $0<\alpha\le 2$ and $\nu>0$. Assume that $\theta_0\in
L^p(\mathbb{R}^N)$ with $p>1$. Then there exists a time $T>0$ such
that the system (\ref{1.1}) has a solution $\theta$, defined in $[0,
T]$, satisfying $\theta\in  L^q([0, T];L^p(\mathbb{R}^N))$ for any
$q>1$.

Further, assume that $\theta_0\in W^{l,p}(\mathbb{R}^N)$ with $l>1,
p>1$. Then there exists a time $T>0$ such that the system
(\ref{1.1}) has a solution $\theta$, defined in $[0, T]$, satisfying
$\partial_x^\beta\theta\in L^q([0, T];L^p(\mathbb{R}^N))$ for any
$q>1$ and $0\le |\beta|\le l$.
\end{theorem}
Secondly, for the sub-critical case $1<\alpha\le 2$ and $\nu>0$, we
have the following global existence and uniqueness results on strong
or smooth solution to the system \eqref{1.1}.
\begin{theorem}\label{thm1.3}
Let $1<\alpha\leq 2$ and $\nu> 0$ and suppose that $\theta_0\geq
0$.
\begin{description}\item[(i)]
 If $\theta_0\in
 H^2(\mathbb{R}^N)\cap L^p(\mathbb{R}^N)(p>\frac N{\alpha-1})$, then there exists a unique
 global solution $\theta$ to (\ref{1.1}) satisfying $\theta\in
 C([0,\infty);H^2(\mathbb{R}^N))$.
\item[(ii)]
 If $\theta_0\in
 H^s(\mathbb{R}^N)\cap L^p(\mathbb{R}^N)(s>0, p>\frac N{\alpha-1})$, then there exists a unique
 global solution $\theta$ to (\ref{1.1}) satisfying $\theta\in
 C([0,\infty);H^s(\mathbb{R}^N))$.
 \end{description}
\end{theorem}
For the critical case  $\alpha=1$, we have the following regularity
result.
\begin{theorem}\label{thm1.4}
Let $\theta(x, t)$ be a solution to system (\ref{1.1}). Then
$\theta$ verifies the level set energy inequalities, i.e., for every
$\lambda>0$
\begin{eqnarray}
 &&\int_{\mathbb{R}^N}\theta_\lambda^2 (t_2, x)dx
 +2\int_{t_1}^{t_2}\int_{\mathbb{R}^N}|\Lambda^\frac{1}{2}\theta_\lambda|^2dxdt
 \leq\int_{\mathbb{R}^N} \theta_\lambda^2(t_1,x)dx,\quad
 0<t_1<t_2,\label{mm0}
\end{eqnarray} where $\theta_\lambda=(\theta-\lambda)_+$. It yields
that for every $t_0$ there exists $\gamma>0$ such that $\theta$ is
bounded in $C^\gamma([t_0,\infty)\times \mathbb{R}^N)$.
\end{theorem}
For the super-critical case $0<\alpha<1$, we have the following
global existence results on weak solutions. A similar result holds
for the case $1\le \alpha\le 2$.
\begin{theorem}\label{thm1.5}
Let $T>0$ be arbitrary.  For every  $\theta_0\in L^2(\mathbb{R}^N),
\theta_0\geq 0$ and $0<\alpha\leq 2$, then there exists at least one
weak solution of the system (\ref{1.1}), satisfying
\begin{equation}\label{007}
  \theta\in L^\infty([0,T]; L^2(\mathbb{R} ^N))\cap L^2([0, T];H^{\frac{\alpha}{2}}(\mathbb{R}^N)).
\end{equation}\end{theorem}
Because the weak solution in Theorem \ref{thm1.5} is not unique, we
try to give a unique criterion on weak solution. We have the
following regularity result for $1<\alpha\le 2$.
\begin{theorem}\label{thm1.6}
Let $T>0$ be arbitrary.  For every  $\theta_0\in L^2(\mathbb{R}^N),
\theta_0\geq 0$ and $1<\alpha\leq 2$, there exists a unique solution
of (\ref{1.1}) such that $\theta\in L^\infty(0, T;
L^2(\mathbb{R}^N))\cap
L^2(0,T;H^{\frac{\alpha}{2}}(\mathbb{R}^N))\cap L^p(0,
T;L^q(\mathbb{R}^N))$ for $q>\frac{N}{\alpha-1}$ and
$\frac{1}{p}+\frac{N}{q\alpha}=1-\frac{1}{\alpha}$.
\end{theorem}
For the sub-critical case $\alpha\in(1,2]$, we have the following
decay rate for the global solution.
\begin{theorem}\label{thm1.7}
Let $\alpha\in(1,2]$ and $N>2$. Assume that $\theta_0\geq 0,
\theta_0\in L^2(\mathbb{R}^N)\cap L^1(\mathbb{R}^N)$ and
$\Lambda\theta_0\in L^2(\mathbb{R}^N)$.  Then the solution
$\theta(x,t)$ to the problem (\ref{1.1}) have the following decay
rate in time:
\begin{description}\item[(i)]\quad
\begin{equation}\label{5.5}
 \|\theta(t)\|_{L^2}\leq C(1+t)^{-\frac
 12(\frac{N+2-2\alpha}{\alpha}-\epsilon)};
\end{equation}
\item[(ii)]\quad
\begin{equation}\label{lp1}
\|\theta(t)\|_{L^p}\le C(1+t)^{-\frac{N(p-2)}{2p\alpha}}, p>2;
\end{equation}
\item[(iii)]
 \begin{equation}\label{5.6}
   \|\nabla\theta(t)\|_{L^2}\leq
   C(1+t)^{-\frac 12(\frac{N+2-2\alpha}{\alpha}-\epsilon)}.
\end{equation}
\end{description} Here $C$ is a positive constant and $\epsilon$ is sufficiently small positive
constant.
\end{theorem}
We also mention that, for the incompressible quasi-geostrophic
equations and the related models, there are a lot of results on the
existence, uniqueness and the regularity (see
\cite{4,1,18,17,2,5,16,12} and therein references). Nonlinear
evolution problems involving the fractal Laplacian describing the
anomalous diffusion, called the $\alpha$-stable L\'{e}vy diffusion,
have been extensively studied in the mathematical and physical
literature (see \cite{6,1,10,18,19} and therein references).

Before ending this section, we give some preliminary Lemmas and
recall some properties of the fractional operator $\Lambda^\alpha$,
which will be used later.

First, we need the following basic calculus inequality (see
\cite{KP98, 9}).
\begin{lemma}\label{le2.2}
For $s\ge 1$ and $1<r<p\leq \infty$,
\begin{eqnarray}
  && \|\Lambda^s(uv)\|_{L^r}\leq
   C(\|u\|_{L^p}\|\Lambda^sv\|_{L^q}+\|v\|_{L^p}\|\Lambda^su\|_{L^q}),\label{lam1}\\
  && \|\Lambda^s(fg)-f\Lambda^sg\|_{L^2}\le C(\|\nabla
f\|_{L^\infty}\|\Lambda^{s-1}g\|_{L^2}+\|g\|_{L^\infty}\|\Lambda^{s}f\|_{L^2})\label{lam2}
\end{eqnarray}
where $\frac{1}{p}+\frac{1}{q}=\frac{1}{r}$ and $C$ is a constant.
\end{lemma}
We also need the following inequality for the Riesz potential (see
\cite{stein}).
\begin{lemma}\label{le2.3}
Assume $1<q<p<+\infty$, $0<\delta<N$ and
$\frac{1}{q}=\frac{1}{p}+\frac{\delta}{N}$. Then there exists a
constant $C>0$ such that
\begin{equation}\label{001}
    \|\Lambda^{-\delta} f\|_{L^p}\leq C\|f\|_{L^q}.
\end{equation}
\end{lemma}
Secondly, we recall the following point-wise estimate and positive
Lemma (see \cite{14, ju, castro}).
\begin{lemma}\label{prole}
Let $s\in [0, 2]$, $\beta\geq -1$ and $\theta\in \mathcal
{S}(\Omega)$, when $\Omega=\mathbb{R}^N$. Then the following
point-wise inequality holds:
\begin{equation}\label{positive}
|\theta(x)|^\beta \theta(x)\Lambda^s\theta(x)\geq
\frac{1}{\beta+2}\Lambda^s|\theta(x)|^{\beta+2}.
\end{equation}
\end{lemma}
\begin{lemma}\label{le003}
Suppose that  $s\in[0,2], x\in \mathbb{R}^N$ and $ \theta,
\Lambda^s\theta\in L^p$, with $p\in (1,+\infty)$. Then
\begin{eqnarray}\int_{\mathbb{R}^N}|\theta|^{p-2}\theta\Lambda^s\theta dx
\geq\frac 2p\int_{\mathbb{R}^N}(\Lambda^{\frac s2}|\theta|^{\frac
p2})^2dx\geq 0.\label{polem}\end{eqnarray}
\end{lemma}
Next, we recall the basic properties of the fractional operator
$\Lambda^\alpha$ (see \cite{stein}) and the Riesz transform.
\begin{lemma}\label{plam}
\begin{eqnarray}&&(i) \quad \Lambda\nabla=\nabla\Lambda ;\nonumber\\
&&(ii) \quad
\Lambda^\alpha\Lambda^\beta=\Lambda^{\alpha+\beta};\nonumber\\
&&(iii) \quad C^{-1}\|\nabla f\|_{L^2}\le\|\Lambda f\|_{L^2}\le C\|\nabla f\|_{L^2};\nonumber\\
&&(iv) \quad \|\mathcal{R}f\|_{L^p}\le C\|f\|_{L^p},
1<p<\infty\nonumber\end{eqnarray} for some positive constant $C$.
\end{lemma}
Finally, we also give another property of the Riesz transform and
its proof.
\begin{prop}
\label{prop hilbert} Let $\phi$ be a continuous function on
$\mathbb{\mathbb{R}}^N$. For any $f\in
\mathcal{S}(\mathbb{R}^N)$$(\mathcal{S}(\mathbb{R}^N)$is the
Schwartz class on $\mathbb{R}^N)$, we have
\begin{equation}
\label{prop hil}
    \int_{\mathbb{R}^N}\phi(x)f(x)\mathcal {R}f(x)\mathd x
    =\frac{C_N}{2}\int_{\mathbb{R}^N}\int_{\mathbb{R}^N}
    \frac{(x-y)[\phi(x)-\phi(y)]}{|x-y|^{N+1}}f(x)f(y)\mathd x\mathd y .
  \end{equation}
where $C_N=(2\pi)^{-N}$.
  \end{prop}
\begin{proof}.
Denote $\D\widetilde{f_\epsilon}(x)=C_N\int_{\mathbb{R}^N,|x-y|\ge
\epsilon}\frac{(x-y) f(y)}{|x-y|^{N+1}}\mathd y$,\,$\D
F_\epsilon(x)=\phi(x)f(x)\widetilde{f_\epsilon}(x)$ and $\D\bar
f(x)=\sup_{\epsilon\ge 0}|\widetilde{f_\epsilon}(x)|$ . It follows
from the singular integral theory of Calderon-Zygmund \cite{CZ56}
that
\[\D\widetilde{f_\epsilon}(x)\rightarrow \mathcal {R}f(x),~~~\mbox{for a.e.}~x\in \mathbb{R}^N\]
and
\[
\|\bar f\|_{L^p(\mathbb{R}^N)}\le C_p\|f\|_{L^p(\mathbb{R}^N)}.
\]
Therefore, we have $F_\epsilon(x)\rightarrow \phi(x)f(x)\mathcal
{R}f(x)$, for a.e.\,$x\in \mathbb{R}^N$ and $|F_\epsilon(x)|\le
G(x),$ where $G(x)=|\phi(x)f(x)|\bar f(x)$ satisfies
\begin{eqnarray}\|G(x)\|_{L^1(\mathbb{R}^N)}&\le&\|\bar f(x)\|_{L^p(\mathbb{R}^N)}\|\phi(x)f(x)\|_{L^q(\mathbb{R}^N)}\nonumber\\
&\le&C_p\|f(x)\|_{L^p(\mathbb{R}^N)}\|\phi(x)f(x)\|_{L^q(\mathbb{R}^N)}<+\infty.
\nonumber
\end{eqnarray}
where $\frac{1}{p}+\frac{1}{q}=1,\,p>1$.\\
Using the Lebesgue Dominated Convergence Theorem, we have
\begin{eqnarray}
\int_{\mathbb{R}^N}\phi(x)f(x)\mathcal {R}(f)\mathd x & = &
\lim_{\epsilon\rightarrow 0}
\int_{\mathbb{R}^N}f(x)\phi(x) \D\widetilde{f_\epsilon}(x) \mathd x \nonumber \\
&=& C_N\lim_{\epsilon\rightarrow 0}
\int_{\mathbb{R}^N}f(x)\phi(x)\int_{\mathbb{R}^N,|x-y|\ge \epsilon}
\frac{(x-y)f(y)}{|x-y|^{N+1}}\mathd y\mathd x . \label{eqn-P1}
\end{eqnarray}
Note that
\begin{eqnarray*}
&~&\int_{\mathbb{R}^N}|f(y)| (\int_{\mathbb{R}^N,|x-y|\ge \epsilon}
\frac{|(x-y)f(x)\phi(x)|}{|x-y|^{N+1}}\mathd x )\mathd y \\
&\leq & \int_{\mathbb{R}^N}|f(y)|  (\int_{\mathbb{R}^N}\frac{2|f(x)\phi(x)|}{\epsilon+|x-y|^N}\mathd x )\mathd y \\
&\leq &2 \|\phi(x)f(x)\|_{L^q} \|(\epsilon+|x|^N)^{-1}\|_{L^p}
\int_{\mathbb{R}^N}|f(y)|\mathd y \\
&= & C\|f(y) \|_{L^1(\mathbb{R}^N)}
\|\phi(x)f(x)\|_{L^q(\mathbb{R}^N)} < \infty ,
\end{eqnarray*}
for each fixed $\epsilon >0$ since $f \in L^1(\mathbb{R}^N)$, $\phi
f \in L^q(\mathbb{R}^N)$ by our assumption, and $C \equiv 2
\|(\epsilon+|x|^N)^{-1}\|_{L^p(\mathbb{R}^N)} < \infty$ for
$p>1(\frac{1}{p}+\frac{1}{q}=1)$. Thus Fubini's Theorem implies that
\begin{eqnarray}
&~&\displaystyle C_N \int_{\mathbb{R}^N}f(x) \phi(x)
\int_{\mathbb{R}^N,|x-y|\ge \epsilon}
\frac{(x-y)f(y)}{|x-y|^{N+1}}\mathd y \mathd x \nonumber\\
&=&C_N \int_{\mathbb{R}^N}\int_{\mathbb{R}^N,|x-y|\ge \epsilon} f(x)
\phi(x) \displaystyle\frac{(x-y)f(y)}{|x-y|^{N+1}}\mathd y \mathd x,
\label{eqn-P2}
\end{eqnarray}
for each fixed $\epsilon > 0$. Furthermore, by renaming the
variables in the integration, we can rewrite $1/2$ of the integral
on the right hand side of \myref{eqn-P2} as follows:
\begin{eqnarray}
&~&\displaystyle\frac{C_N}{2}
\int_{\mathbb{R}^N}\int_{\mathbb{R}^N,|x-y|\ge \epsilon} f(x)f(y)
\frac{(x-y)\phi(x)}{|x-y|^{N+1}}
\mathd y\mathd x\nonumber\\
&=&- \displaystyle\frac{C_N}{2}
\int_{\mathbb{R}^N}\int_{\mathbb{R}^N,|x-y|\ge \epsilon} f(x)f(y)
\frac{(x-y)\phi(y)}{|x-y|^{N+1}} \mathd x\mathd y, \nonumber
\end{eqnarray}
which implies that
\begin{eqnarray}
&~&\displaystyle C_N \int_{\mathbb{R}^N}\int_{\mathbb{R}^N,|x-y|\ge
\epsilon} f(x)f(y) \frac{(x-y)\phi(x)}{|x-y|^{N+1}}
\mathd y\mathd x\nonumber\\
 &=&\displaystyle\frac{C_N}{2} \int_{\mathbb{R}^N}\int_{\mathbb{R}^N,|x-y|\ge \epsilon} f(x)f(y) \frac{(x-y)[\phi(x)-\phi(y)]}{|x-y|^{N+1}}
\mathd x\mathd y . \label{eqn-P3}
\end{eqnarray}
Since $f \in \mathcal {S}^(\mathbb{R}^N)$ and $\phi(x)$ is
continuous on $\mathbb{R}^{N}$, it is obvious that
\[
f(x) f(y) \frac{(x-y)[\phi(x)-\phi(y)]}{|x-y|^{N+1}} \in
L^1(R^{2N}).
\]
Using the Lebesgue Dominated Convergence Theorem, we have
\begin{eqnarray}
&~&\displaystyle\frac{C_N}{2}\lim_{\epsilon\rightarrow 0} \int_{\mathbb{R}^N}\int_{\mathbb{R}^N,|x-y|\ge \epsilon} f(x)f(y) \frac{(x-y)[\phi(x)-\phi(y)]}{|x-y|^{N+1}}\mathd x\mathd y\nonumber\\
&=&\displaystyle\frac{C_N}{2}\int_{\mathbb{R}^N}\int_{\mathbb{R}^N}f(x)f(y)
\frac{(x-y)[\phi(x)-\phi(y)]}{|x-y|^{N+1}}\mathd x\mathd y .
\label{eqn-P4}
\end{eqnarray}
Proposition \ref{prop hilbert} now follows from
\myref{eqn-P1}-\myref{eqn-P4}.
\end{proof}
The rest of this article is organized as follows. In section 2, we
prove the local existence and uniqueness of smooth non-negative
solutions or the strong solution to the system \eqref{1.1} with or
without the dissipation term. Section 3 is devoted to the global
existence and uniqueness of smooth or strong solutions for the
sub-critical and critical cases. In section 4, we prove the global
existence of weak solution and give one Leray-Prodi-Serrin condition
on uniqueness of the strong solution for the sub-critical case. An
example that the non-positive solution to the system \eqref{1.1}
with $\nu=0$ can not be global in time is also given. Finally, we
establish the decay rate of the smooth solution to the system
\eqref{1.1} in the sub-critical case as $t\to\infty$ in section 5.
%%%%%%%%%%%%%%%%%%%%%%%%%%%%%%%%%%%%%%%%%%%%%%%%%%%%%%%%%%%%%%%%%%第二部分
\section{Local existence: proofs of Theorems \ref{thm1.1} and \ref{thm1.2}}
\label{sec:1} In this section, we give proofs of Theorems
\ref{thm1.1} and \ref{thm1.2}.

{\bf The proof of Theorem \ref{1.1}}: For $\alpha=2$ and $\nu>0$,
the existence and uniqueness of local smooth solution is standard.
We will prove our results for the case $0\le\alpha<2$ and $\nu\ge 0$
by using the regularization method. We consider the regularization
system as follows:
\begin{eqnarray}&&\partial_t\theta^\epsilon+u^\epsilon\cdot\nabla\theta^\epsilon+\theta^\epsilon div u^\epsilon
=-\nu\Lambda^\alpha\theta^\epsilon+\epsilon\Delta\theta^\epsilon,
x\in
\mathbb{R}^N, t>0,\label{reg1}\\
&&u^\epsilon=\mathcal{R}\theta^\epsilon, div
u^\epsilon=div\mathcal{R}\theta^\epsilon=\Lambda\theta^\epsilon,x\in
\mathbb{R}^N, t>0,\label{reg2}\\
&&\theta^\epsilon(x, t)=\theta_0(x), x\in
\mathbb{R}^N,\label{reg3}\end{eqnarray} which, using the semigroup
theory, can be re-written into the equivalent integral form:
\begin{eqnarray}\theta^\epsilon(x, t)=e^{\epsilon t\Delta}\theta_0(x)+\int_0^te^{\epsilon (t-\tau)\Delta}(
-u^\epsilon\cdot\nabla\theta^\epsilon-\theta^\epsilon div
u^\epsilon-\nu\Lambda^\alpha\theta^\epsilon)(x, \tau)d\tau, x\in
\mathbb{R}^N, t>0.\label{regint1}
\end{eqnarray} Notice that the singular integral operator
$\Lambda^\alpha, 0\le\alpha<2,$ is of the order $\alpha,
0\le\alpha<2,$ and, hence, the system \eqref{reg1}-\eqref{reg3} is a
parabolic one of the order 2 with nonlocal singular integrals, which
are an operators from $H^s$ to $H^s$ for any $s\ge 0$. Thus, it is
easy to prove, by the standard parabolic theory and using the fact
that $\|u^\epsilon\|_{H^s}\le C\|\theta^\epsilon\|_{H^s},
\|\Lambda\theta^\epsilon\|_{H^s}\le C\|\theta^\epsilon\|_{H^{s+1}}$,
that, for any $\epsilon>0$, there exists $T^\epsilon>0$ such that
the system \eqref{reg1}-\eqref{reg3} has a unique smooth solution
$\theta\in C(0, T^\epsilon; H^s(\mathbb{R}^N))\bigcap C^1(0,
T^\epsilon; H^{s-2}(\mathbb{R}^N))$. Moreover, using the fact that,
if $\theta(x_0, t_0)=\min_{x\in \mathbb{R}^N, t\ge
0}\theta^\epsilon(x, t)$, then
$$\Lambda^\alpha\theta^{\epsilon}(x_0, t_0)=C_\alpha
P.V.\int_{\mathbb{R}^N}\frac{\theta^{\epsilon}(x_0, t_0)
-\theta^{\epsilon}(y, t)}{|x_0-y|^{N+\alpha}}dy\leq 0,$$ it is easy
to prove that, if $\theta_0(x)\ge 0$ in $\mathbb{R}^N$, then
$\theta^\epsilon\ge 0$ in $\mathbb{R}^N\times [0, T^\epsilon)$.

In the following we want to prove that, if $\theta_0\in
H^s(\mathbb{R}^N), s>\frac N2,$ satisfying $\theta_0\ge 0$, then
there exist a time $T_0=T_0(\theta_0)>0$ and a positive constant
$M$, independent of $\epsilon$ such that, for all $\epsilon>0$, the
solution $\theta^\epsilon$ of the system \eqref{reg1}-\eqref{reg3}
satisfies $\theta^\epsilon\ge 0$ and
\begin{eqnarray}\sup_{0\le t\le T_0}\|\theta^\epsilon(x,
t)\|_{H^s(\mathbb{R}^N)}+\sup_{0\le t\le
T_0}\|\partial_t\theta^\epsilon(x,
t)\|_{H^{s-2}(\mathbb{R}^N)}+\int_0^{T_0}\|\theta^\epsilon(x,
t)\|_{H^{s+\frac \alpha2}(\mathbb{R}^N)}dt\le M.\label{regunif1}
\end{eqnarray}
Multiplying the equation (\ref{reg1}) by
$\Lambda^{2s}\theta^\epsilon$ and integrating by parts, we obtain
\begin{eqnarray}&&\frac{1}{2}\frac{d}{dt}\|\Lambda^s\theta^\epsilon\|_{L^2}^2
+\nu\|\Lambda^{s+\frac{\alpha}{2}}\theta^\epsilon\|_{L^2}^2+\epsilon\|\Lambda^{s+1}\theta^\epsilon\|_{L^2}^2\nonumber\\
 &=&-\int  _{\mathbb{R}^N}\Lambda^s(u^\epsilon\cdot\nabla\theta^\epsilon+\theta^\epsilon div u^\epsilon)\cdot\Lambda^s\theta^\epsilon dx \nonumber \\
 &=&- \int  _{\mathbb{R}^N}\Lambda^s(u^\epsilon\cdot\nabla\theta^\epsilon)\Lambda^s \theta^\epsilon dx
 -\int  _{\mathbb{R}^N}\Lambda^s(\theta^\epsilon\Lambda\theta^\epsilon)\Lambda^s \theta^\epsilon  dx \nonumber \\
 &\equiv&I_1+I_2.\label{est}
 \end{eqnarray}
For the first term $I_1$, we have
\begin{eqnarray}
I_{1}&=&- \int
_{\mathbb{R}^N}\Lambda^s(u^\epsilon\cdot\nabla\theta^\epsilon)\Lambda^s
\theta^\epsilon dx\nonumber\\
&=&-\int _{\mathbb{R}^N}
u^\epsilon\cdot\nabla\Lambda^s\theta^\epsilon\Lambda^s
\theta^\epsilon dx-\int _{\mathbb{R}^N}[\Lambda^s(u^\epsilon\cdot
\nabla
\theta^\epsilon)-u^\epsilon\cdot\Lambda^s(\nabla\theta^\epsilon)]\Lambda^s\theta^\epsilon
dx  \nonumber \\
&=&-\int _{\mathbb{R}^N}
u^\epsilon\cdot\nabla\frac{|\Lambda^s\theta^\epsilon|^2}2 dx-\int
_{\mathbb{R}^N}[\Lambda^s(u^\epsilon\cdot \nabla
\theta^\epsilon)-u^\epsilon\cdot\Lambda^s(\nabla\theta^\epsilon)]\Lambda^s\theta^\epsilon
dx  \nonumber\\
&=&\int _{\mathbb{R}^N} \Lambda\theta^\epsilon
\frac{|\Lambda^s\theta^\epsilon|^2}2 dx-\int
_{\mathbb{R}^N}[\Lambda^s(u^\epsilon\cdot \nabla
\theta^\epsilon)-u^\epsilon\cdot\Lambda^s(\nabla\theta^\epsilon)]\Lambda^s\theta^\epsilon
dx  \nonumber \\
&\leq&\frac 12\|\Lambda\theta^\epsilon\|_{L^\infty}\int
_{\mathbb{R}^N} |\Lambda^s\theta^\epsilon|^2
dx+\|\Lambda^s(u^\epsilon\cdot\nabla\theta^\epsilon)
-u^\epsilon\cdot\Lambda^s(\nabla\theta^\epsilon)\|_{L^2}\|\Lambda^s\theta^\epsilon\|_{L^2}\nonumber\\
&\le&\frac 12\|\Lambda\theta^\epsilon\|_{L^\infty}\int
_{\mathbb{R}^N} |\Lambda^s\theta^\epsilon|^2 dx+C(\|\nabla
u^\epsilon\|_{L^\infty}\|\Lambda^{s-1}\nabla\theta^\epsilon\|_{L^2}+\|\nabla\theta^\epsilon\|_{L^\infty}\|\Lambda^s
u^\epsilon\|_{L^2})\|\Lambda^s\theta^\epsilon\|_{L^2}\nonumber \\
&\le&C\|\Lambda\theta^\epsilon\|_{H^{s-1}}\int _{\mathbb{R}^N}
|\Lambda^s\theta^\epsilon|^2 dx+C(\|\nabla
u^\epsilon\|_{H^{s-1}}\|\Lambda^{s-1}\nabla\theta^\epsilon\|_{L^2}+\|\nabla\theta^\epsilon\|_{H^{s-1}}\|\Lambda^s
u^\epsilon\|_{L^2})\|\Lambda^s\theta^\epsilon\|_{L^2}\nonumber\\
%\end{eqnarray}
%\begin{eqnarray}
&\le&C\|\Lambda^s\theta^\epsilon\|_{L^2}\int _{\mathbb{R}^N}
|\Lambda^s\theta^\epsilon|^2 dx+C\|\Lambda^s
\theta^\epsilon\|_{L^{2}}\|\Lambda^{s}\theta^\epsilon\|_{L^2}\|\Lambda^s\theta^\epsilon\|_{L^2}
\le
C\|\Lambda^s\theta^\epsilon\|_{L^2}^3.\label{esti1}\end{eqnarray}
Here $C$ is a positive constant independent of $\epsilon$, and we
have used the fact that $\|\nabla f\|_{H^s}\le C\|\Lambda
f\|_{H^s}$, $\|u^\epsilon\|_{H^s}\le C\|\theta^\epsilon\|_{H^s}$ for
$u^\epsilon=\mathcal{R}\theta^\epsilon$ and the inequality
\eqref{lam2}.

For the second term $I_2$, using the fact that $\theta^\epsilon\ge
0$, the pointwise estimate \eqref{positive} with $\beta=0$ for the
operator $\Lambda$ and the inequality \eqref{lam2}, we have
\begin{eqnarray}\label{cc}
I_{2}&=&-\int _{\mathbb{R}^N}
\theta^\epsilon\Lambda(\Lambda^s\theta^\epsilon)\Lambda^s
\theta^\epsilon dx-\int
_{\mathbb{R}^N}[\Lambda^s(\theta^\epsilon\Lambda
\theta^\epsilon)-\theta^\epsilon\Lambda^s(\Lambda\theta^\epsilon)]\Lambda^s\theta^\epsilon
dx\nonumber \\
&\leq&- \int _{\mathbb{R}^N}
\theta^\epsilon\Lambda\frac{|\Lambda^s\theta^\epsilon|^2}2
dx+C(\|\Lambda\theta^\epsilon\|_{L^\infty}+\|\nabla\theta^\epsilon\|_{L^\infty})\|\Lambda^s\theta^\epsilon\|_{L^2}^2 \nonumber \\
&=&\frac 12\int _{\mathbb{R}^N} \Lambda\theta^\epsilon
|\Lambda^s\theta^\epsilon|^2
dx+C(\|\Lambda\theta^\epsilon\|_{L^\infty}+\|\nabla\theta^\epsilon\|_{L^\infty})\|\Lambda^s\theta^\epsilon\|_{L^2}^2 \nonumber \\
&\le&C(\|\Lambda\theta^\epsilon\|_{L^\infty}+\|\nabla\theta^\epsilon\|_{L^\infty})\|\Lambda^s\theta^\epsilon\|_{L^2}^2 \nonumber \\
&\le&C\|\Lambda^s\theta^\epsilon\|_{L^2}^3. \label{esti2}
\end{eqnarray} Combining \eqref{est} with \eqref{esti1} and
\eqref{esti2}, one have
\begin{eqnarray}\label{est1}
&&\frac{1}{2}\frac{d}{dt}\|\Lambda^s\theta^\epsilon\|_{L^2}^2+\nu\|\Lambda^{s+\frac{\alpha}{2}}\theta^\epsilon\|_{L^2}^2
+\epsilon\|\Lambda^{s+1}\theta^\epsilon\|_{L^2}^2\le
C\|\Lambda^s\theta^\epsilon\|_{L^2}^3,
\end{eqnarray}
which claims that there exist a time $T_0>0$ and a constant
$M(T)>0$, independent of $\epsilon$, such that $\sup_{0\le t\le
T_0}\|\Lambda^s\theta^\epsilon(\cdot,
t)\|_{L^2}+\int_0^{T_0}\|\Lambda^{s+\frac
\alpha2}\theta^\epsilon(\cdot, t)\|_{L^2}^2dt\le M(T_0)$. Then, by
the equations \eqref{reg1} and \eqref{reg2}, the uniform estimate
for $\partial_t\theta^\epsilon$ with respect to $\epsilon$ can be
obtained easily.

Now combining the above estimates with the compactness argument,
letting $\epsilon\to 0$, we obtain the desired results on the local
smooth solutions to the system \eqref{1.1}. Moreover, it follows
from \eqref{est1} that, if $0<T<\infty$, $T$ is the maximal
existence time of the solution to the system \eqref{1.1}, then
$\int_0^T\|\theta(\cdot, t)\|_{H^s}dt=\infty$ or
$\int_0^T(\|\Lambda\theta\|_{L^\infty}+\|\nabla\theta\|_{L^\infty}+\|\nabla
u\|_{L^\infty})dt=\infty$.

Next we give the proof of uniqueness. Let $T>0$ be the maximal
existence time of the solution to the system \eqref{1.1}, and assume
that $\theta_1, \theta_2\in C([0, T^*]; H^{s}), T^*<T,$ are two
solutions to (\ref{1.1}) with velocities $u_1=\mathcal{R}\theta_1$
and $u_2=\mathcal{R}\theta-2$, respectively, and the same initial
data $\theta_0\in H^s$. Denote $\theta=\theta_1-\theta_2$ and
$u=u_1-u_2$, then we have
\begin{equation}\label{bb}
\partial_t\theta+u\cdot\nabla\theta_1+u_2\cdot\nabla\theta+\theta
divu_1+\theta_2 divu=-\nu\Lambda^\alpha \theta .
\end{equation}
Multiplying both hand side of the equation (\ref{bb}) by $\theta$,
we have
\begin{equation}
\frac{1}{2}\frac{d}{dt}\|\theta\|_{L^2}^2+\int _{\mathbb{R}^N}
u\cdot\nabla \theta_1\theta dx+\int _{\mathbb{R}^N}
u_2\cdot\nabla\theta\theta dx+\int _{\mathbb{R}^N} \theta\Lambda
\theta_1\theta dx+\int _{\mathbb{R}^N} \theta_2\Lambda\theta\theta
dx =-\int _{\mathbb{R}^N} \nu\Lambda^\alpha\theta\theta dx.
\end{equation}
We can calculate
\begin{eqnarray}
&&\frac{1}{2}\frac{d}{dt}\|\theta\|_{L^2}^2+\nu\|\Lambda^{\frac{\alpha}{2}}\theta\|_{L^2}^2  \nonumber\\
&\leq& -\int _{\mathbb{R}^N} u\cdot\nabla\theta_1\theta
dx-\frac{1}{2}\int _{\mathbb{R}^N} u_2\cdot\nabla|\theta|^2dx-\int
_{\mathbb{R}^N} \theta\theta\Lambda\theta_1dx-\frac{1}{2}\int
_{\mathbb{R}^N}
\theta_2\Lambda\theta^2dx\nonumber\\
&\leq& C(\|\nabla\theta_1\|_{L^\infty}\|\theta\|_{L^2}^2+\|\nabla
u_2\|_{L^\infty}\|\theta\|_{L^2}^2+\|\Lambda\theta_1\|_{L^\infty}\|\theta\|_{L^2}^2
+\frac{1}{2}\|\Lambda\theta_2\|_{L^\infty}\|\theta\|_{L^2}^2)\nonumber\\
&\leq&
C(\|\theta_1\|_{H^s}+\|u_2\|_{H^s}+\|\theta_1\|_{H^s}+\|\theta_2\|_{H^s})\|\theta\|_{L^2}^2
\nonumber\\
&\le&C(\|\theta_1\|_{H^s}+\|\theta_2\|_{H^s})\|\theta\|_{L^2}^2.\label{uni1}
\end{eqnarray}
Here we use $s>\frac{N}{2}+1$ and $\theta_2\ge 0$. Applying the
Gronwall's inequality to the inequality (\ref{uni1}) and using the
fact that $\|\theta_1(t)\|_{H^s}$ and $\|\theta_2(t)\|_{H^s}$ is
bounded for $t\in [0, T^*]$, we can obtain the desired uniqueness
result.

{\bf Proof of Theorem \ref{thm1.2}} We will prove Theorem
\ref{thm1.2} by using the fixed point principle by constructing
contraction mapping.

We re-write the system \eqref{1.1} into the equivalent integral
system
\begin{equation} \theta (x,
t)=G_\alpha(t)\theta_0(x)-\int_0^tG_\alpha(t-\tau)div(u\theta)(\tau)d\tau,\label{thetaint}\end{equation}
where $G_\alpha(t)$ is given by the Fourier transform
$\widehat{G_\alpha(t)}=e^{-\nu|\xi|^\alpha t}$, and satisfies the
following boundedness \cite{2,5,7}.
\begin{lemma}\label{lem2.1}
Assume $1\leq p\leq q\leq \infty$. Then, for any $t>0$, the
operators $G_\alpha(t)$ and $\nabla G_\alpha(t)$ are bounded from
$L^p$ to $L^q$. Furthermore, we have,  for any $f\in L^p$, that
\begin{eqnarray}&&\|G_\alpha(t)f\|_{L^q}\leq C
t^{-\frac{2}{\alpha}(\frac{1}{p}-\frac{1}{q})}\|f\|_{L^p},\label{pq1}\\
&&\|\nabla G_\alpha(t)f\|_{L^q}\leq C
t^{-\frac{1}{\alpha}-\frac{2}{\alpha}(\frac{1}{p}-\frac{1}{q})}\|f\|_{L^p},\label{pq2}\end{eqnarray}
where $C$ is a constant depending only on $\alpha, p$ and $q$.

Further, assume that $u$ and $\theta$ are in $L^q([0,T];
L^p(\mathbb{R}^N))$, then the operator $A(u,\theta)\equiv
\int_0^t\nabla G_\alpha(t-\tau)
 (u\theta)d\tau$ is bounded in $L^q([0,T];L^p(\mathbb{R}^N))$ with
 \begin{eqnarray}\|A(u,\theta)\|_{L^q([0,T];L^p(\mathbb{R}^N))}\leq C\|u\|_{L^q([0,T];L^p(\mathbb{R}^N))}\cdot
 \|\theta\|_{L^q([0,T];L^p(\mathbb{R}^N))},\label{pq3}\end{eqnarray} where $C$ is a constant depending only on
 $\alpha, p$ and $q$.
\end{lemma}
For $l=0$, define the space $X=\{\theta\in L^q([0,T];L^p):
\|\theta\|_X\le M<\infty\}$ with the norm
$\|\cdot\|_X=\|\cdot\|_{L^q([0,T];L^p)}$, and define the mapping $F$
mapping $\theta\in X$ to $F(\theta)$ by
\begin{equation} F(\theta) (x,
t)=G_\alpha(t)\theta_0(x)-\int_0^tG_\alpha(t-\tau)div(u\theta)(\tau)d\tau\label{map1}\end{equation}
with the velocity $u=\mathcal{R}\theta$. In the following, we will
prove that

(i) If $\theta\in X$, then $F(\theta)\in X$;

(ii) For any $\theta, \tilde{\theta}\in X$, then
$\|F(\theta)-F(\tilde{\theta})\|_X\le \frac
12\|\theta-\tilde{\theta}\|_X$ for some $T>0$.

In fact, by using \eqref{pq1} in Lemma \ref{lem2.1}, we can easily
conclude that $F(0)=G_\alpha (t)\theta_0$ is bounded in
$L^q([0,T];L^p(\mathbb{R}^N))$, i.e.,
\begin{eqnarray}\label{32}
  \|G_\alpha(t)\theta_0\|_{L^q([0,T];L^p)}&=& [\int_0^T\|G_\alpha(t)\theta_0\|_{L^p}^qdt]^{\frac{1}{q}} \nonumber \\
  &\leq&[\int_0^T\|\theta_0\|_{L^p}^qdt]^{\frac{1}{q}} \nonumber \\
  &\leq&\|\theta_0\|_{L^p}[\int_0^Tdt]^{\frac{1}{q}} \nonumber \\
   &\leq& CT^{\frac 1q}\|\theta_0\|_{L^p}.
\end{eqnarray}
Now we choose $M=3CT^{\frac 1q}\|\theta_0\|_{L^p}$ sufficiently
small by using $\alpha>1$ and letting $T$ sufficiently small, and
hence we have
$\|F(0)\|_{L^q([0,T];L^p)}=\|G_\alpha(t)\theta_0\|_{L^q([0,T];L^p)}\leq
\frac{M}{3}$.

Let $\theta$ and $\tilde{\theta}$ be any two elements of $X$, where
$u$ and $\tilde{u}$ be the velocities corresponding to $\theta$ and
$\tilde{\theta}$, respectively. Then, using \eqref{pq3} in Lemma
\ref{lem2.1}, we have
\begin{eqnarray}\label{33}
  &&\|F(\theta)-F(\tilde{\theta})\|_{L^q([0,T];L^p)}\nonumber\\
  &=& \|\int_0^t\nabla G(t-\tau)(u\theta)(\tau)d\tau
  -\int_0^t\nabla G(t-\tau)(\tilde{u}\tilde{\theta})(\tau)d\tau\|_{L^q([0,T];L^p)}\nonumber\\
   &=& \|A(u,\theta)-A(\tilde{u},\tilde{\theta})\|_{L^q([0,T];L^p)}\nonumber \\
   &=& \|A(u-\tilde{u},\theta)+A(\tilde{u},\theta-\tilde{\theta})\|_{L^q([0,T];L^p)}\nonumber\\
  &\leq& \|A(u-\tilde{u},\theta)\|_{L^q([0,T];L^p)}+ \|A(\tilde{u},\theta-\tilde{\theta})\|_{L^q([0,T];L^p)}\nonumber\\
   &\leq&
  C\|u-\tilde{u}\|_{L^q([0,T];L^p)}\|\theta\|_{L^q([0,T];L^p)}
  +C\|\tilde{u}\|_{L^q([0,T];L^p)}\|\theta-\tilde{\theta}\|_{L^q([0,T];L^p)}.
\end{eqnarray}
Because $u$ and $\tilde{u}$ are Riesz transforms of $\theta$ and
$\tilde{\theta}$, respectively, the classical Calderon-Zygmund
singular integral estimates imply that
\begin{equation}\label{a9a}
 \|u\|_{L^q([0,T];L^p)}\leq C\|\theta\|_{L^q([0,T];L^p)},
\end{equation}
and
\begin{equation}\label{b9a}
 \|\tilde{u}\|_{L^q([0,T];L^p)}\leq
 C\|\tilde{\theta}\|_{L^q([0,T];L^p)}.
\end{equation}
Substituting  inequalities (\ref{a9a}) and (\ref{b9a}) into
(\ref{33}), we get
 \begin{eqnarray}
  \|F(\theta)-F(\tilde{\theta})\|_{L^q([0,T];L^p)}&\leq& C(\|\theta\|_{L^q([0,T];L^p)}
  +\|\tilde{\theta}\|_{L^q([0,T];L^p)})\|\theta-\tilde{\theta}\|_{L^q([0,T];L^p)} \nonumber\\
    &\leq&
    CM\|\theta-\tilde{\theta}\|_{L^q([0,T];L^p)}.\label{the3}
 \end{eqnarray}
Hence, using \eqref{the3} and letting $M$ to be small enough, we
have
 \begin{eqnarray}
 \|F(\theta)\|_{L^q([0,T];L^p)}&=& \|F(\theta)-F(0)+F(0)\|_{L^q([0,T];L^p)}  \nonumber\\
    &\leq& \|F(\theta)-F(0)\|_{L^q([0,T];L^p)}+\|F(0)\|_ {L^q([0,T];L^p)} \nonumber\\
  &\leq& CM\|\theta\|_{L^q([0,T];L^p)}+\frac{M}{3}\nonumber\\
  &\leq&CM^2+\frac{M}{3}\nonumber\\
  &\leq&M\nonumber
 \end{eqnarray}
and
\begin{eqnarray}
    \|F(\theta)-F(\tilde{\theta})\|_{L^q([0,T];L^p)}&\leq&\frac{1}{2}
    (\|\theta-\tilde{\theta}\|_{L^q([0,T];L^p)}.\nonumber
\end{eqnarray}
By the contracting mapping principle, there exists a unique function
$\theta\in X$ such that $F(\theta)=\theta$ and, hence, there exists
a time $T>0$ such that the system \eqref{1.1} has a solution
$\theta\in L^q([0,T];L^p)$.

For $l>0$, define the space $X=\{\theta\in L^q([0,T];W^{l,p}):
\|\theta\|_X\le M<\infty\}$ with the norm
$\|\cdot\|_X=\|\cdot\|_{L^q([0,T];W^{l,p})}$. Similar to the proof
of the case $l=0$, we can obtain the local existence of the smooth
solution to the system \eqref{1.1}.

This ends the proof of Theorem \ref{thm1.2}.

%%%%%%%%%%%%%%%%%%%%%%%%%%%%%%%%%%%%%%%%%%%%%%%%%%%%%%%%%%%第二部分
\section{Global existence of strong  and smooth solution: proofs of Theorems \ref{thm1.3} and \ref{thm1.4} } \label{sec:2}
In this section, we will prove the global existence of strong or
smooth solution to the system \eqref{1.1} for the sub-critical and
critical cases $1\le \alpha\le 2$ by the careful energy methods.

{\bf Proof of Theorem \ref{thm1.3}}: If we assume that $s>\frac
N2+1$, then the local existence can be guaranteed by Theorem
\ref{thm1.1}. For general $s>0$, we can prove the local existence of
the strong or smooth solution by the fixed point theory as in the
proof of Theorem \ref{thm1.2}. To prove the global existence, it
suffices to establish the a priori estimates globally in time. This
is divided into the following three steps.

{\bf Step 1: $L^p$-estimate and Maximum principle}

When $\alpha=2$, the result is obvious. We only need consider the
case $1<\alpha<2$.

We notice the fact that, if $\theta_0(x)\geq 0$, then
$\theta(x,t)\geq 0$ in $\overline {\Omega_T}=\mathbb{R}^N\times
(0,T])$.

Multiplying both sides of equation (\ref{1.1})$_1$ by $\theta^p(x,
t)$ and integrating the resulting equation in $\mathbb{R}^N$, one
get
\begin{eqnarray}\nonumber
    \frac{1}{p+1}\frac{d}{dt}\int_{\mathbb{R}^N}\theta^{p+1}dx&=&\int_{\mathbb{R}^N}-div(R(\theta)\theta)\theta^pdx
   -\nu\int_{\mathbb{R}^N}(-\Delta)^{\frac{\alpha}{2}}\theta\cdot\theta^pdx\nonumber\\
    &=&\int_{\mathbb{R}^N} R(\theta)\theta\cdot p\theta^{p-1}
    \nabla \theta dx-\nu\int_{\mathbb{R}^N}\Lambda^{\alpha}\theta\cdot\theta^pdx \nonumber \\
   &=& \frac{p}{p+1}\int_{\mathbb{R}^N} R(\theta)\nabla \theta^{p+1}dx
   -\nu\int _{\mathbb{R}^N}\theta^p\Lambda^{\alpha}\theta dx \nonumber \\
  &=& -\frac{p}{p+1}\int_{\mathbb{R}^N} \theta^{p+1}\Lambda \theta dx
  -\nu\int_{\mathbb{R}^N}\theta^p\Lambda^{\alpha}\theta dx\le 0
  \nonumber
\end{eqnarray} with the aid of \eqref{polem} in Lemma \ref{le003}. Hence, we have
 \begin{equation}
\int_{\mathbb{R}^N}\theta^{p+1}(x)dx\leq
\int_{\mathbb{R}^N}\theta_0^{p+1}(x)dx, \quad\forall p>1,
\nonumber
 \end{equation}
i.e.,
 \begin{equation}\label{lp1}
\|\theta\|_{L^p}\leq \|\theta_0\|_{L^p}.\nonumber
 \end{equation}
 In particular, if we take $p\rightarrow \infty$, we get
 \begin{equation}\label{lp2}
\|\theta\|_{L^\infty}\leq \|\theta_0\|_{L^\infty}.\nonumber
 \end{equation}
{\bf Step 2 : A priori estimate in $H^s$ for the case $s=2$}

Multiplying both sides of equation $(\ref{1.1})_1$ by
$\Lambda^4\theta$ and taking the inner product with the resulting
equation in $L^2$, we have
\begin{eqnarray}\label{20}
\frac{1}{2}\frac{d}{dt}\|\Lambda^2\theta\|_{L^2}^2&=&-\int_{\mathbb{R}^N}\Lambda^{2+\frac{\alpha}{2}}
\theta\Lambda^{2-\frac{\alpha}{2}}
div(u\theta)dx-\nu\|\Lambda^{2+\frac{\alpha}{2}}\theta\|_{L^2}^2 \nonumber\\
&\leq&\|\Lambda^{2+\frac{\alpha}{2}}\theta\|_{L^2}\|\Lambda^{2+1-\frac{\alpha}{2}}(u\theta)\|_{L^2}-\nu\|\Lambda^{2+\frac{\alpha}{2}}\theta\|_{L^2}^2,
\end{eqnarray}
where we have used the H$\ddot{o}$lder inequality and the calculus
inequality
$\|\Lambda^{2-\frac{\alpha}{2}}div(u\theta)\|_{L^2}\le\|\Lambda^{2+1-\frac{\alpha}{2}}
(u\theta)\|_{L^2}$.

Using the inequalities for the Calderon-Zygmund type singular
integrals on $u=\mathcal {R}\theta$, we have
\begin{equation}\label{22}
    \|u\|_{L^p}\leq C\|\theta\|_{L^p},\quad
    \|\Lambda^{3-\frac{\alpha}{2}}u\|_{L^q}\leq
    C\|\Lambda^{3-\frac{\alpha}{2}}\theta\|_{L^q},\quad  1<p,   q<+\infty.
\end{equation}
By using \eqref{lam1} in Lemma \ref{le2.2} and (\ref{22}), we have
 \begin{eqnarray}\label{21}
   \|\Lambda^{2+1-\frac{\alpha}{2}}(u\theta)\|_{L^2}&\leq& C(\|u\|_{L^p}\|\Lambda^{3-\frac{\alpha}{2}}
   \theta\|_{L^q}+\|\theta\|_{L^p}\|\Lambda^{3-\frac{\alpha}{2}}u\|_{L^q}) \nonumber\\
    &\leq&C(\|\theta\|_{L^p}\|\Lambda^{3-\frac{\alpha}{2}}
   \theta\|_{L^q}+\|\theta\|_{L^p}\|\Lambda^{3-\frac{\alpha}{2}}\theta\|_{L^q}) \nonumber \\
  &\leq&
  C\|\theta\|_{L^p}\|\Lambda^{3-\frac{\alpha}{2}}\theta\|_{L^q},
 \end{eqnarray}
where $\frac{1}{p}+\frac{1}{q}=\frac{1}{2}, p,q>2$.

Putting (\ref{21}) into (\ref{20}), we have
\begin{eqnarray}\label{23}
  \frac{1}{2}\frac{d}{dt}\|\Lambda^2\theta\|_{L^2}^2 \leq
  C\|\theta\|_{L^p}\|\Lambda^{2+\frac{\alpha}{2}}\theta\|_{L^2}
  \|\Lambda^{3-\frac{\alpha}{2}}\theta\|_{L^q}-\nu\|\Lambda^{2+\frac{\alpha}{2}}\theta\|_{L^2}^2.
\end{eqnarray}
Using the Lemma \ref{le2.3}, we have
\begin{equation}\label{24}
\|\Lambda^{3-\frac{\alpha}{2}}\theta\|_{L^q}\leq
C\|\Lambda^{3-\frac{\alpha}{2}+\delta}\theta\|_{L^2},
\end{equation}
where $\frac{1}{q}=\frac{1}{2}-\frac{\delta}{N}$.

Now we take $\delta=\frac{N}{p}, p>\frac{N}{\alpha-1}\geq 2
(1<\alpha\le 2)$, and therefore $1+\delta<\alpha$. Then apply the
fractional type Gagliardo-Nirenberg inequality
\begin{equation}\label{gn}
\|\Lambda^{3-\frac{\alpha}{2}+\delta}\theta\|_{L^2}\leq
\|\Lambda^{2+\frac{\alpha}{2}}\theta\|_{L^2}^a\|\Lambda^2\theta\|_{L^2}^{1-a}
\end{equation}
with the parameter $a=\frac{2-\alpha+2\delta}\alpha<1$.

Putting (\ref{23}), (\ref{24}) and (\ref{gn}) together,  and using
the Young's inequality, we obtain
\begin{eqnarray}
  \frac{1}{2} \frac{d}{dt}\|\Lambda^2\theta\|_{L^2}^2&\leq& C \|\theta\|_{L^p}\|\Lambda^{2
  +\frac{\alpha}{2}}\theta\|_{L^2}^{a+1}\|\Lambda^2\theta\|_{L^2}^{1-a}
  -\nu\|\Lambda^{2+\frac{\alpha}{2}}\theta\|_{L^2}^2\nonumber \\
  &\leq& C \|\theta\|_{L^p}^{\frac2{1-a}}\|\Lambda^{2}\theta\|_{L^2}^2
  -\frac{\nu}2\|\Lambda^{2+\frac{\alpha}{2}}\theta\|_{L^2}^2,\nonumber
\end{eqnarray} which, together with $L^p$ estimate of $\theta$ in Step 1,
gives
\begin{eqnarray}
 \frac{1}{2}\frac{d}{dt}\|\Lambda^2\theta\|_{L^2}^2\le C(\nu,
 \|\theta_0\|_{L^p})\|\Lambda^2\theta\|_{L^2}^2,\nonumber
\end{eqnarray}
which gives
\begin{equation}\label{28}
\|\Lambda^2\theta\|_{L^2}(t)\leq \|\Lambda^2\theta_0\|_{L^2}e^{Ct}.
\end{equation}

{\bf Step 3: A priori estimate in $H^s$ for the case $s>0$}

Multiplying both sides of equation $(\ref{1.1})_1$ by
$\Lambda^{2s}\theta$ and taking the inner product with the
resulting equation in $L^2$, we have
\begin{eqnarray}\label{20s}
\frac{1}{2}\frac{d}{dt}\|\Lambda^s\theta\|_{L^2}^2&=&-\int_{\mathbb{R}^N}\Lambda^{s+\frac{\alpha}{2}}
\theta\Lambda^{s-\frac{\alpha}{2}}
div(u\theta)dx-\nu\|\Lambda^{s+\frac{\alpha}{2}}\theta\|_{L^2}^2 \nonumber\\
&\leq&\|\Lambda^{s+\frac{\alpha}{2}}\theta\|_{L^2}\|\Lambda^{s+1-\frac{\alpha}{2}}(u\theta)\|_{L^2}
-\nu\|\Lambda^{s+\frac{\alpha}{2}}\theta\|_{L^2}^2,
\end{eqnarray}
where we have used the H\"{o}lder inequality and the calculus
inequality
$\|\Lambda^{s-\frac{\alpha}{2}}div(u\theta)\|_{L^2}\le\|\Lambda^{s+1-\frac{\alpha}{2}}
(u\theta)\|_{L^2}$.

Using the inequalities for the Calderon-Zygmund type singular
integrals $\mathcal {R}$
\begin{equation}
    \|u\|_{L^p}\leq C\|\theta\|_{L^p},\quad
    \|\Lambda^{s+1-\frac{\alpha}{2}}u\|_{L^q}\leq
    C\|\Lambda^{s+1-\frac{\alpha}{2}}\theta\|_{L^q}, 1<p, q<\infty
    \nonumber
\end{equation}
 and Lemma \ref{le2.2}, we have
 \begin{eqnarray}\label{21s}
   \|\Lambda^{s+1-\frac{\alpha}{2}}(u\theta)\|_{L^2}&\leq& C(\|u\|_{L^p}\|\Lambda^{s+1-\frac{\alpha}{2}}
   \theta\|_{L^q}+\|\theta\|_{L^p}\|\Lambda^{s+1-\frac{\alpha}{2}}u\|_{L^q}) \nonumber\\
    &\leq&C(\|\theta\|_{L^p}\|\Lambda^{s+1-\frac{\alpha}{2}}
   \theta\|_{L^q}+\|\theta\|_{L^p}\|\Lambda^{s+1-\frac{\alpha}{2}}\theta\|_{L^q}) \nonumber \\
  &\leq&
  C(\|\theta\|_{L^p}\|\Lambda^{s+1-\frac{\alpha}{2}}\theta\|_{L^q}),
 \end{eqnarray}
where $\frac{1}{p}+\frac{1}{q}=\frac{1}{2}, p,q>2$.

Putting (\ref{21s}) into (\ref{20s}), we have
\begin{eqnarray}\label{23s}
  \frac{1}{2}\frac{d}{dt}\|\Lambda^s\theta\|_{L^2}^2 \leq
  C\|\theta\|_{L^p}\|\Lambda^{s+\frac{\alpha}{2}}\theta\|_{L^2}
  \|\Lambda^{s+1-\frac{\alpha}{2}}\theta\|_{L^q}-\nu\|\Lambda^{s+\frac{\alpha}{2}}\theta\|_{L^2}^2.
\end{eqnarray}
Using the Lemma \ref{le2.3}, we have
\begin{equation}\label{24s}
\|\Lambda^{s+1-\frac{\alpha}{2}}\theta\|_{L^q}\leq
C\|\Lambda^{s+1-\frac{\alpha}{2}+\delta}\theta\|_{L^2},
\end{equation}
where $\frac{1}{q}=\frac{1}{2}-\frac{\delta}{N}$.

Now we take $\delta=\frac{N}{p}, p>\frac{N}{\alpha-1}\geq 2
(1<\alpha\le 2)$, and therefore $1+\delta<\alpha$. Then apply the
fractional type Gagliardo-Nirenberg inequality
\begin{equation}\label{25}
\|\Lambda^{s+1-\frac{\alpha}{2}+\delta}\theta\|_{L^2}\leq
\|\Lambda^{s+\frac{\alpha}{2}}\theta\|_{L^2}^a\|\Lambda^s\theta\|_{L^2}^{1-a}
\end{equation}
with the parameter $a=\frac{2-\alpha+2\delta}\alpha<1$.

Putting (\ref{23s}), (\ref{24s}) and (\ref{25}) together, and
using the Young's inequality, we obtain
\begin{eqnarray}\label{26s}
  \frac{1}{2} \frac{d}{dt}\|\Lambda^s\theta\|_{L^2}^2&\leq& C
  \|\theta\|_{L^p}\|\Lambda^{s
  +\frac{\alpha}{2}}\theta\|_{L^2}^{a+1}\|\Lambda^s\theta\|_{L^2}^{1-a}
  -\nu\|\Lambda^{s+\frac{\alpha}{2}}\theta\|_{L^2}^2\nonumber \\
  &\leq& C \|\theta\|_{L^p}^{\frac2{1-a}}\|\Lambda^{s}\theta\|_{L^2}^2
  -\frac{\nu}2\|\Lambda^{s+\frac{\alpha}{2}}\theta\|_{L^2}^2,\nonumber
\end{eqnarray} which, together with $L^p$ estimate on $\theta$ in Step 1,
gives
\begin{eqnarray}
 \frac{1}{2}\frac{d}{dt}\|\Lambda^s\theta\|_{L^2}^2\le C(\nu,
 \|\theta_0\|_{L^p})\|\Lambda^s\theta\|_{L^2}^2,\nonumber
\end{eqnarray}
which gives
\begin{equation}\label{28s}
\|\Lambda^s\theta\|_{L^2}(t)\leq
\|\Lambda^s\theta_0\|_{L^2}e^{Ct}.
\end{equation}

Using the a priori estimates (\ref{28}), (\ref{28s}) and the
standard extension argument we can conclude the global existence
result. The proof of Theorem \ref{thm1.3} is complete.

{\bf Proof of Theorem \ref{thm1.4}} The proof is analogous to the
critial dissipative quasi-geostrophic equation that is shown in
\cite{caffarelli}. We give two key points of the proof. First, we
have
\begin{eqnarray}
 &&\int_{\mathbb{R}^N}\theta_\lambda^2 (t_2, x)dx
 +2\int_{t_1}^{t_2}\int_{\mathbb{R}^N}|\Lambda^\frac{1}{2}\theta_\lambda|^2dxdt
 +2\int_{t_1}^{t_2}\int_{\mathbb{R}^N}div (u\theta_\lambda)\theta_\lambda
 dxdt\nonumber\\
 &\leq&\int_{\mathbb{R}^N} \theta_\lambda^2(t_1,x)dx,
 0<t_1<t_2.\label{m1}
\end{eqnarray} Next, we
only need to show the term $
\int_{t_1}^{t_2}\int_{\mathbb{R}^N}div(u\theta_\lambda)\theta_\lambda
dx dt $ is positive. In fact, by the direct calculation, we have
\begin{eqnarray}\label{m2}
 \int_{t_1}^{t_2}\int_{\mathbb{R}^N}div(u\theta_\lambda)\theta_\lambda dx
 dt&=&\int_{t_1}^{t_2}\int_{\mathbb{R}^N}div
 u\cdot\theta_\lambda^2dxdt+\int_{t_1}^{t_2}\int_{\mathbb{R}^N}u\cdot\nabla
 \theta_\lambda\cdot\theta_\lambda dxdt\nonumber\\
 &=&\int_{t_1}^{t_2}\int_{\mathbb{R}^N}div
 u\cdot\theta_\lambda^2dxdt-\int_{t_1}^{t_2}\int_{\mathbb{R}^N}divu\cdot
\frac{|\theta_\lambda|^2}{2}dxdt\nonumber\\
&=&\frac{1}{2}\int_{t_1}^{t_2}\int_{\mathbb{R}^N}\Lambda
 \theta_\lambda\cdot\theta_\lambda^2 dxdt\geq 0.
\end{eqnarray}
Here we have used the relationship $u=\mathcal {R} \theta$ and $div
u=div\mathcal {R} \theta=\Lambda \theta$. Combining \eqref{m1} and
\eqref{m2}, we obtain \eqref{mm0}. Then utilize the same strategy as
\cite{caffarelli,19} to finish the proof of Theorem \ref{thm1.4}.
%%%%%%%%%%%%%%%%%%%%%%%%%%%%%%%%%%%%%%%%%%%%%%%%%%%%%%%%%%%%%%%%%%%%%%%%%%%%%%%%%%%%%%%%%%%%%%%%%%%%第四部分
\section{Global existence of the weak solution: proofs of Theorems \ref{thm1.5} and \ref{thm1.6}}
\label{sec:4} In this section, we will prove Theorems \ref{thm1.5}
and \ref{thm1.6} by employing the vanishing viscosity method used in
\cite{14, castro}. We consider the general case $0<\alpha\le 2$.

{\bf Definition 4.1.} A solution $\theta(x, t)$ is called the weak
solution to system (\ref{1.1}), if for any smooth function $\phi\in
C_0^\infty([0, \tau]\times \mathbb{R}^N)$, it satisfies
\begin{eqnarray*}\label{004}
&&\int_{\mathbb{R}^N}\theta(x,t)\phi(x,t)dx-\int_{\mathbb{R}^N}\theta_0(x)\phi(x,0)dx
+\int_0^\tau\int_{\mathbb{R}^N}[-\theta(x,t)\partial_t\phi(x,t)\\
&&-u\theta\cdot\nabla\phi(x,t)+\nu\theta(x,t)\Lambda^\alpha\phi(x,t)]dxdt=0,
\end{eqnarray*}
where the velocity $u=\mathcal {R}\theta$.

Let $\varepsilon>0$ be a small parameter and we will approximate
problem (\ref{1.1})  by considering the regularized system of
(\ref{1.1}) with a small viscosity term
\begin{eqnarray}\label{008}
\left\{\begin{array}{lll} \frac{\partial\theta_\varepsilon}{\partial
t}+div(u_\varepsilon\theta_\varepsilon)
+\nu\Lambda^\alpha\theta_\varepsilon=\varepsilon\Delta\theta_\varepsilon,\\
u_\varepsilon=\mathcal {R}\theta_\varepsilon,\\
\theta_\varepsilon(x,0)=\theta_0^\varepsilon.\end{array}\right.
\end{eqnarray}
for $0<\varepsilon\leq 1,
\theta_0^\varepsilon=\psi_\varepsilon\ast\theta_0,
 \psi_\varepsilon(x)=\varepsilon^{-N}\psi(\frac{x}{\varepsilon})$ and $\psi$  satisfying
$$\psi\geq 0, \quad \psi   \in   C_0^\infty(\mathbb{R}^N) \quad \quad
and \quad \quad
 \|\psi\|_{L^1}=1.$$
For any fixed $\varepsilon>0$, by the standard parabolic theory, as
in the proof of Theorem \ref{thm1.1}, we can prove the following
global existence results on the smooth solution to the regularized
system \eqref{008}.
\begin{prop}\label{tha2.4}
For any $\varepsilon>0$ and for any $\tau>0$, there exists a unique
solution $\theta_\varepsilon$ of (\ref{008}) satisfying
$\theta_\varepsilon\in C([0,\tau];$
$H^s(\mathbb{R}^N))(s>\frac{N}{2}+1)$. Moreover, if $\theta_0\ge 0$,
then $\theta^\varepsilon(x, t)\ge 0$ .
\end{prop} We want to establish the a priori estimates for
$\theta^\varepsilon$ with respect to $\varepsilon$, and then to
perform the limit $\lim_{\varepsilon\to 0}\theta^\varepsilon=\theta$
in the sense of weak convergence, and to verify that the limit
function $\theta$ is a weak solution of the system \eqref{1.1} in
the sense of Definition 4.1.

We multiply both sides of equations $(\ref{008})_1$ by
$\theta_\varepsilon$ to get
\begin{equation}\label{0011}
\frac{1}{2}\frac{d}{dt}\|\theta_\varepsilon\|_{L^2}^2+\nu\|\Lambda^{\frac{\alpha}{2}}\theta_\varepsilon\|_{L^2}^2
+\varepsilon\|\nabla\theta_\varepsilon\|_{L^2}^2\leq
\int_{\mathbb{R}^N}\mathcal
{R}(\theta_\varepsilon)\theta_\varepsilon\cdot\nabla
\theta_\varepsilon
dx=-\frac{1}{2}\int_{\mathbb{R}^N}\theta_\varepsilon^2\Lambda\theta_\varepsilon\leq
0,
\end{equation}
where we have used \eqref{polem} in Lemma \ref{le003} and $\theta\ge
0$.

Then we integrate (\ref{0011}) in time to get
\begin{eqnarray}\label{0012}
   \| \theta_\varepsilon (\tau)\|_{L^2}^2+2\nu\int_0^\tau\|\Lambda^{\frac{\alpha}{2}}\theta_\varepsilon(s)\|_{L^2}^2ds\leq
   \|\theta_0\|_{L^2}^2, \quad \forall \tau.
\end{eqnarray}
In particular, we obtain
\begin{equation}\label{0013}
   \theta_\varepsilon\in C([0,\tau]; {L^2}({\mathbb{R}^N})), \quad
   \sup_{0\le t\le\tau}\|\theta_\varepsilon\|_{L^2(\mathbb{R}^N)}\le
    \|\theta_0\|_{L^2(\mathbb{R}^N)},
   and\quad \max\limits_{0\leq t\leq
   \tau}\|\theta_\varepsilon(t)\|_{L^2}^2\leq \|\theta_0\|_{L^2}^2.
\end{equation} Using $u^\varepsilon=\mathcal{R}^\varepsilon$ and
$L^2$ boundedness of the Riesz transform, one get
\begin{equation}\label{u100}
\|u_\varepsilon\|_{L^2(\mathbb{R}^N)}\le
C\|\theta_\varepsilon\|_{L^2(\mathbb{R}^N)}\le
    C\|\theta_0\|_{L^2(\mathbb{R}^N)}.
\end{equation} Next we pass to the limit $\varepsilon\rightarrow 0$ in (\ref{008})
by using the Aubin-Lions compactness lemma.

First of all, by the previous a priori estimate as in  (\ref{0013}),
we obtain $\theta_\varepsilon\in C([0,\tau]; L^2(\mathbb{R}^N))$ and
\begin{eqnarray}\max\{\|\theta_\varepsilon(t)\|_{L^2}:0\le t\le\tau\}\le M<\infty. \label{comp1}\end{eqnarray}
Secondly, we want to prove that, for any $\phi\in
C_0^\infty(\mathcal{R}^N)$, $\{\phi\theta_\varepsilon\}$ is
uniformly Lipschitz in the interval of time $[0, \tau]$ with respect
to the space $H^{-p}$ with $p>\frac N2+2$, i.e.
\begin{equation}\label{comp2}
 \|\phi\theta_\varepsilon(t_2)-\phi\theta_\varepsilon(t_1)\|_{H^{-p}}\le
 C|t_2-t_1|, \quad 0\le t_1, t_2\le \tau
\end{equation} for some positive constant $C>0$.

Because $\theta_\varepsilon$ is a strong solution of (\ref{008}) and
is continuous, it follows that
\begin{equation}\label{0015}
 \|\phi\theta_\varepsilon(t_2)-\phi\theta_\varepsilon(t_1)\|_{H^{-p}}=\|\int_{t_1}^{t_2}\phi\frac{d}{dt}\theta_\varepsilon dt\|_{H^{-p}}\leq \max\limits_{t_1\leq t\leq t_2}\{M(t)\}(t_2-t_1),
\end{equation}
where
$$M(t)=\|div(\phi \mathcal {R}(\theta_\varepsilon)\theta_\varepsilon)\|_{H^{-p}}+\|\nabla\phi\mathcal {R}(\theta_\varepsilon)\theta_\varepsilon\|_{H^{-p}}
+\nu\|\phi\Lambda^\alpha
\theta_\varepsilon\|_{H^{-p}}+\varepsilon\|\phi
\Delta\theta_\varepsilon\|_{H^{-p}}. $$ Using the Sobolev's
imbedding theorem and using \eqref{0013} and \eqref{100}, we have
 \begin{eqnarray}\label{0016}
    \|\nabla\phi\mathcal {R}(\theta_\varepsilon)\theta_\varepsilon\|_{H^{-p}}
 &\leq& C(p)\|\widehat{\nabla\phi\mathcal {R}(\theta_\varepsilon)}\theta_\varepsilon\|_{L^\infty}\nonumber\\
  &\leq& C(p)\|{\nabla\phi}\|_{L^\infty}\|{\mathcal {R}(\theta_\varepsilon)}\theta_{\varepsilon}\|_{L^{1}} \nonumber\\
  &\leq&C(p, \phi)\|\theta_{\varepsilon}\|_{L^2}^2\nonumber
 \\ &\leq& C(p,\phi)\|\theta_0\|_{L^2}^2.
\end{eqnarray}
Similarly, we have
\begin{eqnarray}\label{0017}
  \|div ((\phi \mathcal {R}(\theta_\varepsilon)\theta_\varepsilon))\|_{H^{-p}}
   &\leq& \|\phi \mathcal {R}(\theta_\varepsilon)\theta_\varepsilon\|_{H^{1-p}} \nonumber\\
   &\leq& C(\phi, p)\|\theta_\varepsilon\|_{L^2}^2\nonumber\\
   &\leq& C(\phi, p)\|\theta_0\|_{L^2}^2.
\end{eqnarray}
Applying the convolution property of the Fourier transforms, we have
\begin{eqnarray}
  |\int_{\mathbb{R}^N} \widehat{\phi(y)} |\xi-y|^\alpha\widehat{\theta_\varepsilon}(\xi-y)dy|
  &\leq& C\int_{\mathbb{R}^N}
  (|\xi|^\alpha+|y|^\alpha)|\hat{\phi}(y)||\widehat{\theta_\varepsilon}(\xi-y)|dy\nonumber\\
   &\leq&C(1+|\xi|^\alpha)\|\phi\|_{H^\alpha}\|\theta_\varepsilon(0)\|_{L^2},\nonumber
\end{eqnarray}
which gives
\begin{eqnarray}\label{0019}
  \|\phi\Lambda^\alpha\theta_\varepsilon\|_{H^{-p}} &\leq& C(\phi)\|\theta_\varepsilon\|_{L^2}
  (\int_{\mathbb{R}^N}\frac{(1+|\xi|^\alpha)^2}{(1+|\xi|^2)^p}d\xi)^{\frac{1}{2}}\nonumber\\
   &\leq& C(p,\phi)\|\theta_\varepsilon\|_{L^2}\le
   C(p,\phi)\|\theta_0\|_{L^2}.
\end{eqnarray}
Similarly, we obtain
\begin{equation}\label{0020}
    \|\phi\Delta \theta_\varepsilon\|_{H^{-p}}\leq C(p,\phi)\|\theta_0\|_{L^2}.
\end{equation}
Putting \eqref{0015} together with \eqref{0016}-\eqref{0020}, we
obtain \eqref{comp2}.

From \eqref{comp1}-\eqref{comp2}, conditions (i) and (ii) of the
Aubin-Lions lemma \cite{castro} are satisfied. Therefore, there
exists a subsequence and a function $\theta\in C([0,\tau];
L^2(\mathbb{R}^N))$ such that
\begin{equation}\label{0025}
\theta_\varepsilon\rightharpoonup \theta \quad in \quad
L^2(\mathbb{R}^N)\quad a.e. \text{ t and }   \max\limits_{0\leq
t\leq
\tau}\|\phi\theta_\varepsilon(t)-\phi\theta(t)\|_{H^{-p}}\rightarrow
0.
\end{equation}
We take the limit in the weak formulation of the problem (\ref{008})
\begin{eqnarray}
  &&\int_{\mathbb{R}^N}\theta_\varepsilon(x,\tau)\phi dx
  -\int_{\mathbb{R}^N}\theta_\varepsilon(x,0)\phi(x,0)dx+
  \int_0^\tau\int_{\mathbb{R}^N}[-\theta_\varepsilon(x,t)\partial_t\phi(x,t)-u_\varepsilon
  \theta_\varepsilon\cdot\nabla\phi(x,t)\nonumber\\
  &&+\nu\theta_\varepsilon(x,t)\Lambda^\alpha\phi(x,t)-\varepsilon\theta_\varepsilon(x,t)\Delta\phi(x,t)]
  dxdt=0,\nonumber
\end{eqnarray}
and let $\varepsilon\rightarrow 0$, we get
\begin{eqnarray}\label{0022}
  &&\int_{\mathbb{R}^N}\theta(x,\tau)\phi(x,t)dx-\int_{\mathbb{R}^N}\theta_0(x,t)\phi(x,0)
  +\int_0^\tau\int_{\mathbb{R}^N}
 [\theta(x,t)\partial_t\phi(x,t)+\nu\theta(x,t)\Lambda^\alpha\phi(x,t)]dx
 dt\nonumber\\
  &&+\lim\limits_{\varepsilon\rightarrow 0}\int_0^\tau\int_{\mathbb{R}^N}\theta_\varepsilon u_\varepsilon\nabla\phi(x,t)dxdt=0,
\end{eqnarray}
Now we rewrite the last term in the left hand side of \eqref{0022}
as follows:
\begin{eqnarray}
  &&\int_0^\tau\int_{\mathbb{R}^N}\theta_\varepsilon u_\varepsilon\cdot\nabla\phi(x,t)dx dt\nonumber\\
  &\leq&   \int_0^\tau\int_{\mathbb{R}^N}(\theta_\varepsilon-\theta)u_\varepsilon\cdot\nabla \phi dx dt+\int_0^\tau
  \int_{\mathbb{R}^N}\theta u_\varepsilon\cdot\nabla\phi dx dt\nonumber\\
  &\equiv& I_1+I_2.\label{0024}
\end{eqnarray}
The first term $I_1$ can be estimated by
\begin{eqnarray}
  |I_1|&=&|\int_0^\tau\int_{\mathbb{R}^N}(\theta_\varepsilon-\theta)u_\varepsilon\cdot\nabla \phi dx dt|\nonumber\\
  &\leq& \int_0^\tau\|u_\varepsilon\|_{H^{\frac{\alpha}{2}}}\|(\theta_\varepsilon-\theta)\nabla \phi\|_{H^{-\frac{\alpha}{2}}}dt \nonumber\\
   &\leq&\max\limits_{0\leq t\leq\tau}\|(\theta_\varepsilon-\theta)\nabla\phi\|_{H^{-\frac{\alpha}{2}}}\int_0^\tau(\|u_\varepsilon\|_{L^2}
   +\|\Lambda^{\frac{\alpha}{2}}u_\varepsilon\|_{L^2})dt \nonumber\\
    &\leq&C(\tau)\|\theta_0\|_{L^2}\max\limits_{0\leq t\leq T}\|(\theta_\varepsilon-\theta)\cdot\nabla\phi\|_{H^{-\frac{\alpha}{2}}}\rightarrow
    0,\label{i1}
\end{eqnarray}
where we have used the fact (\ref{0012}) and (\ref{0025}).

Then, by \eqref{0024}, (\ref{0025}) and \eqref{i1}, we have
\begin{equation}\nonumber
\lim\limits_{\varepsilon\rightarrow0}\int_0^\tau\int_{\mathbb{R}^N}u_\varepsilon\theta_\varepsilon\cdot\nabla\phi
dxdt=\int_0^\tau\int_{\mathbb{R}^N}u\theta\cdot\nabla\phi dxdt,
\end{equation}
which, together with \eqref{0022}, completes the proof of Theorem
\ref{thm1.4}.

{\bf Proof of Theorem \ref{thm1.6}} Let $\theta_1$ and $\theta_2$ be
two solutions to the system \eqref{1.1} with the velocities $u_1$
and $u_2$, respectively. The difference $\theta=\theta_1-\theta_2$
satisfies
\begin{equation}\label{60}
\partial _t\theta +div(u_1
\theta)+div(u\theta_2)+\nu\Lambda^\alpha\theta=0,
\end{equation}
where $u=u_1-u_2$. Clearly, $\theta(x,0)=0$.

Now multiply both sides of (\ref{60}) by $\Lambda^{-1}\theta$ and
integrate by parts, one get
\begin{equation}\label{65}
\frac{d}{dt}\|\Lambda^{-\frac{1}{2}}\theta\|_{L^2}^2+\nu\|\Lambda^{-\frac{1}{2}}
(\Lambda^{\frac{\alpha}{2}})\theta\|_{L^2}^2\leq
|\int_{\mathbb{R}^N}(u_1\theta)\cdot(\nabla(\Lambda^{-1}\theta))dx|
+|\int_{\mathbb{R}^N}(u\theta_2)\cdot(\nabla(\Lambda^{-1}\theta))dx|.
\end{equation}
By the H\"{o}lder inequality, we have
\begin{eqnarray}\label{61}
  |\int_{\mathbb{R}^N}(u_1\theta)\cdot(\nabla(\Lambda^{-1}\theta))dx| &\leq&\|u_1\|_{L^q}\|\theta\|_{L^p}
  \|\nabla(\Lambda^{-1}\theta)\|_{L^p} \nonumber\\
    &\leq&\|\theta_1\|_{L^q}\|\theta\|_{L^p}\|\nabla(\Lambda^{-1}\theta)\|_{L^p},
\end{eqnarray}
and
\begin{eqnarray}\label{62}
|\int_{\mathbb{R}^N}(u\theta_2)\cdot(\nabla(\Lambda^{-1}\theta))dx|
&\leq&\|u\|_{L^p}\|\theta_2\|_{L^q}
  \|\nabla(\Lambda^{-1}\theta)\|_{L^p} \nonumber\\
    &\leq&\|\theta_2\|_{L^q}\|\theta\|_{L^p}\|\nabla(\Lambda^{-1}\theta)\|_{L^p},
\end{eqnarray}
where $\frac{1}{q}+\frac{2}{p}=1$.

Thanks to the fact
$\widehat{\nabla(\Lambda^{-1})}=(\widehat{\mathcal {R}_1},
\widehat{\mathcal {R}_2},\cdots, \widehat{\mathcal {R}_N})$, we get
from \eqref{61} and \eqref{62} that
\begin{equation}\label{63}
  |\int_{\mathbb{R}^N}(u_1\theta)\cdot(\nabla(\Lambda^{-1}\theta))dx|
+|\int_{\mathbb{R}^N}(u\theta_2)\cdot(\nabla(\Lambda^{-1}\theta))dx|\leq
C(\|\theta_1\|_{L^q}+\|\theta_2\|_{L^q})\|\theta\|^2_{L^p},
\end{equation}
Using \eqref{001} in Lemma \ref{le2.3}, one get
\begin{equation}\label{64}
\|\theta\|_{L^p}\leq
C\|\Lambda^{\frac{N}{2q}}\theta\|_{L^2}=C\|\Lambda^{\frac{N}{2q}+\frac{1}{2}}(\Lambda^{-\frac{1}{2}}\theta)\|_{L^2}
\leq\|\Lambda^{-\frac{1}{2}}\theta\|_{L^2}^r\|\Lambda^{\frac{\alpha}{2}}(\Lambda^{-\frac{1}{2}}\theta)\|_{L^2}^{1-r}.
\end{equation}
In the last inequality,  we have used fractional type
Gagliardo-Nirenberg  inequality with
$\frac{\alpha-1}{N}-\frac{1}{q}=\frac{\alpha r}{N}$.

Substituting  (\ref{63}) and (\ref{64}) into (\ref{65}), we conclude
\begin{eqnarray}
  &&\frac{d}{dt}\|\Lambda^{-\frac{1}{2}}\theta\|_{L^2}^2 +\nu\|\Lambda^{-\frac{1}{2}}(\Lambda^{\frac{\alpha}{2}}
  )\theta\|_{L^2}^2\nonumber\\
  &\leq& C(\|\theta_1\|_{L^q}+\|\theta_2\|_{L^q})
  \|\Lambda^{-\frac{1}{2}}\theta\|_{L^2}^{2r}\|\Lambda^{\frac{\alpha}{2}}
  (\Lambda^{-\frac{1}{2}}\theta)\|_{L^2}^{2(1-r)} \nonumber\\
   &\leq&\frac{C}{\nu}(\|\theta_1\|_{L^q}+\|\theta_2\|_{L^q})^{\frac{1}{r}}\|\Lambda^{-\frac{1}{2}}\theta\|_{L^2}^2
   +\frac{\nu}{2}\|\Lambda^{\frac{\alpha}{2}}(\Lambda^{-\frac{1}{2}}\theta)\|_{L^2}^2,\nonumber
\end{eqnarray}
i.e.,
\begin{equation}
\frac{d}{dt}\|\Lambda^{-\frac{1}{2}}\theta\|_{L^2}^2\leq\frac{C}{\nu}(\|\theta_1\|_{L^q}
+\|\theta_2\|_{L^q})^{\frac{1}{r}}\|\Lambda^{-\frac{1}{2}}\theta\|_{L^2}^2.
\end{equation}
The Gronwall inequality implies that $\theta=0$ and we complete the
proof.

\begin{rem}\label{rem1} Furthermore,
when $\nu=0$, if a compactly supported initial condition
$\theta_0\le 0, \not=0$ has a sufficiently big integral $M=-\int
\theta_0dx$, then the non-positive solution to the Cauchy problem
(\ref{1.1}) can not be global in a time. This can be proven by
borrowing the idea used in \cite{BW24}.
\end{rem}
In fact, let $\Theta=-\theta$ and
$\int_{\mathbb{R}^N}|x|^2\Theta(x,t)dx=w(t)$, then we have
\begin{eqnarray}\label{3.10}
\frac{d}{dt}w(t)&=&\int_{\mathbb{R}^N}|x|^2\nabla\cdot(\mathcal
{R}\Theta(x,t)\Theta(x,t))dx  \nonumber\\
&=&-\int_{\mathbb{R}^N}2x(\mathcal{R}(\Theta)\Theta)(x,t)dx
\nonumber\\
%&=&\int_{\mathbb{R}^N}2x\Theta(x,t)P.V.\int_{\mathbb{R}^N}\frac{(y-x)\Theta(y,t)}{|x-y|^{N+1}}dy
%dx \nonumber\\
%&=&\int_{\mathbb{R}^N}\lim\limits_{\varepsilon\rightarrow
%0}\int_{\mathbb{R}^N: |y-x|\ge\varepsilon}\frac{2x(y-x)\Theta(x,t)\Theta(y,t)}{|x-y|^{N+1}} dy dx
%\nonumber\\
%&=&\lim\limits_{\varepsilon\rightarrow
%0}\int\int_{\mathbb{R}^N\times\mathbb{R}^N:|x-y|\ge\varepsilon}\frac{2x(y-x)\Theta(x,t)\Theta(y,t)}{|x-y|^{N+1}}dy dx
%\nonumber\\
%&=&\lim\limits_{\varepsilon\rightarrow
%0}\int\int_{\mathbb{R}^N\times\mathbb{R}^N:|x-y|\ge\varepsilon}\frac{2y(x-y)\Theta(x,t)\Theta(y,t)}{|x-y|^{N+1}}dx dy
%\nonumber\\
% &=&\lim\limits_{\varepsilon\rightarrow
%0}\int\int_{\mathbb{R}^N\times\mathbb{R}^N:|x-y|\ge\varepsilon}\frac{\Theta(x,t)\Theta(y,t)}{|x-y|^{N-1}}
%[\frac{x(y-x)
%+y(x-y)}{|x-y|^2}]dx dy  \nonumber\\
&=&-C_N\int_{\mathbb{R}^N}\int_{\mathbb{R}^N}\frac{(x-y)\cdot
(x-y)\Theta(x,t)\Theta(y,t)}{|x-y|^{N-1}}dx
dy.\nonumber\\
&=&-C_N\int_{\mathbb{R}^N}\int_{\mathbb{R}^N}\frac{\Theta(x,t)\Theta(y,t)}{|x-y|^{N-1}}dx
dy.
\end{eqnarray}
where we have used the property \eqref{prop hil} of the Riesz
transform.

Let $M=-\int_{\mathbb{R}^N}\theta_0dx=\int_{\mathbb{R}^N}\Theta dx,
J_N=\int_{\mathbb{R}^N}\int_{\mathbb{R}^N}|x-y|^{-(N-1)}\Theta(x,t)\Theta(y,t)dx
dy$, then we have
\begin{eqnarray}
M^2&=&\int_{\mathbb{R}^N}\int_{\mathbb{R}^N}\Theta(x,t)\Theta(y,t)dx
dy\nonumber\\
&=&\int_{\mathbb{R}^N}\int_{\mathbb{R}^N}(\Theta(x,t)\Theta(y,t))^{\frac{N-1}{N+1}}|x-y|^{\frac{2(N-1)}{N+1}}
(\Theta(x,t)\Theta(y,t))^{\frac{2}{N+1}}|x-y|^{-\frac{2(N-1)}{N+1}}dx
dy
\nonumber\\
&=&(\int_{\mathbb{R}^2}\int_{\mathbb{R}^N}\Theta(x,t)\Theta(y,t)|x-y|^2dx
dy)^{\frac{N-1}{N+1}}(\int_{\mathbb{R}^N}\int_{\mathbb{R}^N}\Theta(x,t)\Theta(y,t)|x-y|^{-(N-1)}dx
dy)^{\frac{2}{N+1}}
\nonumber\\
&=&\int_{\mathbb{R}^N}\int_{\mathbb{R}^N}\Theta(x,t)\Theta(y,t)(|x|^{2}-x\cdot
y-y\cdot x+|y|^2)dx dy)^\frac{N-1}{N+1}J_N^{\frac{2}{N+1}}
\nonumber\\
&=&(2Mw-2|\int_{\mathbb{R}^2}x\Theta(x,t)dx|^2)^{\frac{N-1}{N+1}}J_N^{\frac{2}{N+1}}
\nonumber\\
&\le&(2Mw)^{\frac{N-1}{N+1}}J_N^{\frac{2}{N+1}},\nonumber
\end{eqnarray}
which implies
\begin{equation}\label{3.12}
2^{-\frac{N-1}{2}}M^{\frac{N+3}{2}}w^{-\frac{N-1}{2}}\leq J_N,
\end{equation}
Combing the above inequalities (\ref{3.10}) and (\ref{3.12}), we
know
\begin{equation}
\frac{dw}{dt}\leq
-C_N2^{-\frac{N-1}{2}}M^{\frac{N+3}{2}}w^{-\frac{N-1}{2}},
\label{inen1}
\end{equation}
If we assume the right-hand side of the inequality \eqref{inen1} is
strictly negative for $t=0$, then it is always strictly negative for
some finite $t>0$. Hence $w(t)$ will be negative for some finite
$t$, which is a contradiction with $w(t)\ge 0$. This completes the
proof of Remark \ref{rem1}.
%%%%%%%%%%%%%%%%%%%%%%%%%%%%%%%%%%%%%%%%%%%%%%%%%%%%%%%%%%%%%%%%%%%%%%%%%%%%%%%%%%%%%%%%%%%%%%%%%%%%第五部分
\section{Asymptotic behavior: The proof of Theorem \ref{thm1.7}} \label{sec:5}
In this section, we prove Theorem \ref{thm1.7} by using Fourier
splitting method, which was used first by Schonbek
\cite{schonbek2,schonbek3} and then used in \cite{23, zhou} to
obtain decay rate in the context of the usual quasi-geostrophic
equations. It should be pointed out that the present proofs could be
extended to the system for the case $\alpha\in (0,2]$ provided there
were on a priori bound of the derivatives of the solutions in the
space $L^2$. For the global weak solution, the similar decay rate
estimate can be also obtained by using the retarded mollification
technique used in \cite{caffarelli, 23, schonbek2}.

{\bf Proof of the Theorem \ref{thm1.7}}: We will establish the decay
estimate by employing the Fourier splitting method.

First we claim that $\theta$ satisfies the following a priori
estimate
\begin{equation}\label{5.24}
 |\hat{\theta}(\xi,t)|\leq
 \|\theta_0\|_{L^1}+|\xi|\int_0^t\|\theta(\tau)\|_{L^2}^2d\tau.
\end{equation}
In fact, we have from (\ref{1.1})
\begin{equation}\label{5.25}
  \partial_t\hat{\theta}+\nu |\xi|^{2\alpha}\hat{\theta}=-\widehat{div(u\theta)},
\end{equation}
and we estimate the right-hand side of (\ref{5.25}) as follows
\begin{equation}\label{5.26}
|
-\widehat{div(u\theta)}|=|\xi\widehat{u\theta}|=|\xi||\widehat{u\theta}|=|\xi|\|u\|_{L^2}\|\theta\|_{L^2}\leq
|\xi|\|\theta(t)\|_{L^2}^2.
\end{equation}
After integrating (\ref{5.25}) and using \eqref{5.26}, we obtain
\eqref{5.24}.

Now we want to obtain the decay estimate $\|\theta(t)\|_{L^2}$.
Multiplying both sides of (\ref{1.1}) by $\theta(t)$ and integrating
in $\mathbb{R}^N$, one get
\begin{equation}\label{5.28}
    \frac{1}{2}\frac{d}{dt}\int_{\mathbb{R}^N}| \theta |^2dx+\nu\int_{\mathbb{R}^N}
    |\Lambda^{\frac{\alpha}{2}}\theta|^2dx=-\int_{\mathbb{R}^N}div(u\theta)\theta
    dx=-\frac 12\int_{\mathbb{R}^N} \theta^2\Lambda\theta
    dx\leq 0,
\end{equation}
which gives, by the Plancherel's theorem, that
\begin{eqnarray}
&&\|\theta\|_{L^2}\le \|\theta_0\|_{L^2},\label{l2}\\
&&\frac{d}{dt}\int_{\mathbb{R}^N}|\hat{\theta}|^2d\xi+2\nu\int_{\mathbb{R}^N}|\xi|^\alpha|\hat{\theta}|^2d\xi=
 -\int_{\mathbb{R}^N}\widehat{div(u\theta)\bar{\widehat{\theta}}}d\xi\leq
 0.\label{5.29}
\end{eqnarray}
Let introduce $B(t)=\{\xi\in \mathbb{R}^N; |\xi|\leq M(t)\}$ with
$M(t)>0$ to be determined appropriately below and $B(t)^c$ is the
complement of $B(t)$. By \eqref{l2}, we can estimate the second term
in the left hand side of (\ref{5.29})
\begin{equation}
\int_{\mathbb{R}^N}|\xi|^\alpha|\hat{\theta}|^2d\xi\geq
\int_{B(t)^c}|\xi|^\alpha|\hat{\theta}|^2d\xi\geq
M^\alpha(t)\int_{B(t)^c}|\hat{\theta}|^2d\xi
=M^\alpha(t)\int_{\mathbb{R}^N}|\hat{\theta}|^2d\xi
-M^\alpha(t)\int_{B(t)}|\hat{\theta}|^2d\xi,\label{5.30}
\end{equation}
Combining \eqref{5.29} and \eqref{5.30}, using \eqref{5.24} and
\eqref{l2}, we get
\begin{eqnarray}
&&\frac{d}{dt}\int_{\mathbb{R}^N}|\hat{\theta}|^2d\xi+2M^\alpha(t)\nu\int_{\mathbb{R}^N}|\hat{\theta}|^2d\xi\nonumber\\
&\leq&
CM^\alpha(t)\int_0^{M(t)}\big(\|\theta_0\|_{L^1}+r\int_0^t\|\theta(\tau)\|_{L^2}^2d\tau\big)^2r^{N-1}dr\nonumber\\
&\le&CM^\alpha(t)\int_0^{M(t)}\big(\|\theta_0\|_{L^1}^2+r^2t\int_0^t\|\theta(\tau)\|_{L^2}^4d\tau\big)
r^{N-1}dr.\label{5.31}
\end{eqnarray}
Integrating (\ref{5.31}), we get
\begin{eqnarray}\label{5.32}
  &&e^{2\nu\int_0^tM^\alpha(\tau)d\tau}\int_{\mathbb{R}^N}|\hat{\theta}|^2d\xi \nonumber\\
  &\leq& \|\theta_0\|_{L^2}^2+C\int_0^te^{2\nu\int_0^sM^\alpha(\tau)d\tau}\big(\|\theta_0\|_{L^1}^2M^{N+\alpha}(s)\nonumber\\
  &&+sM^{N+2+\alpha}(s)\int_0^s\|\theta(\tau)\|_{L^2}^4d\tau\big)ds.
\end{eqnarray}
Now we take $M^\alpha(t)=\frac{1}{2\beta\nu(t+1)}$ and thus
$e^{2\nu\int_0^tM^\alpha(\tau)d\tau}=(1+t)^{\frac{1}{\beta}}$. From
(\ref{5.32}) and \eqref{l2}, we get
\begin{eqnarray}
&&(1+t)^\frac{1}{\beta}\int_{\mathbb{R}^N}|\hat{\theta}|^2d\xi \nonumber\\
 &\leq&\|\theta_0\|_{L^2}^2+C\int_0^t(1+s)^{\frac{1}{\beta}}\{\|\theta_0\|_{L^1}^2
 (\frac 1{2\alpha\nu(1+s)})^{\frac{N+\alpha}\alpha}\nonumber\\
&&+s(\frac{1}{2\alpha\nu(s+1)})^{\frac{{N+2+\alpha}}{\alpha}}\int_0^s\|\theta(\tau)\|_{L^2}^4d\tau\}
ds\nonumber\\
&\leq&\|\theta_0\|_{L^2}^2
+C\int_0^t\|\theta_0\|_{L^1}^2(\frac{1}{2\beta\nu})^{\frac{{N+\alpha}}{\alpha}}(1+s)^{-\frac{N+\alpha}{\alpha}
 +\frac{1}{\beta}}ds\nonumber\\
 &&+C\int_0^t(\frac{1}{2\beta\nu})^{\frac{{N+2+\alpha}}{\alpha}}(1+s)^{-\frac{N+2+\alpha}{\alpha}
 +\frac{1}{\beta}+2}ds\|\theta_0\|_{L^2}^4.\label{inebeta}
\end{eqnarray} Since $N>2$ and $1\le \alpha \le 2$, we take $\frac
1\beta=\frac{N+2-2\alpha}\alpha-\epsilon$ for some small
$\epsilon>0$, and hence $\frac 1\beta-\frac{N+2-2\alpha}\alpha<0$
and $\frac1\beta-\frac N\alpha<0$. Thus, from \eqref{inebeta}, we
obtain
\begin{eqnarray}
&&(1+t)^{-\frac{1}{\beta}}\|\theta(t)\|_{L^2}^2\nonumber\\
&\leq&
C\|\theta_0\|_{L^2}^2+C\|\theta_0\|_{L^1}^2(\frac{1}{2\beta\nu})^{\frac{{N+\alpha}}{\alpha}}\frac
1{\frac
N\alpha-\frac1\beta}+C(\frac{1}{2\beta\nu})^{\frac{{N+2+\alpha}}{\alpha}}\|\theta_0\|_{L^2}^4
\frac1{\frac{N+2-2\alpha}\alpha-\frac 1\beta},\nonumber
\end{eqnarray}
which gives the following decay rate in time
\begin{equation}\label{5.34}
 \|\theta(t)\|_{L^2}^2\leq C(1+t)^{-(\frac{N+2-2\alpha}{\alpha}-\epsilon)}
\end{equation}
for some $\epsilon$ sufficiently small. Here the constant $C$
depends upon $L^1$ and $L^2$ norms of $\theta_0$.

Next we obtain the decay estimate on $\|\theta(t)\|_{L^p}, p>2$ by
using the method used in \cite{castro}.

Multiplying both sides of (\ref{1.1}) by
$|\theta(t)|^{p-2}\theta(t)$, integrating in $\mathbb{R}^N$ and
applying \eqref{le003} in Lemma \ref{le003}, one get
\begin{equation}\label{5.28p}
    \frac{1}{p}\frac{d}{dt}\int_{\mathbb{R}^N}| \theta |^pdx+\nu\frac 2p\int_{\mathbb{R}^N}
    |\Lambda^{\frac{\alpha}{2}}|\theta|^{\frac
    p2}|^2dx\le-\int_{\mathbb{R}^N}div(u\theta)|\theta|^{p-2}\theta
    dx=-\frac 12\int_{\mathbb{R}^N} |\theta|^p\Lambda\theta
    dx\leq 0.
\end{equation} Using Gagliardo-Nirenberg inequality, we have
\begin{equation}\label{5.29p}
    \frac{d}{dt}\int_{\mathbb{R}^N}| \theta |^pdx\le -2\nu C(\int_{\mathbb{R}^N}
    |\theta|^{\frac{pN}{N-\alpha}}dx)^{\frac{N-\alpha}N}
\end{equation} with $C$ depending on $\alpha$ and $N$. By
interpolation we get
\begin{equation}\label{5.30p}\|\theta\|_{L^p}\le\|\theta\|_{L^2}^{1-\gamma}(\int_{\mathbb{R}^N}
    |\theta|^{\frac{pN}{N-\alpha}}dx)^{\gamma\frac{N-\alpha}{pN}},
    \gamma=\frac{N(p-2)}{N(p-2)+2\alpha}.\end{equation} Putting
    \eqref{5.30p} into \eqref{5.29p}, we have
\begin{equation}\label{5.31p}
    \frac{d}{dt}\|\theta(t)\|_{L^p}^p+2C\nu\|\theta\|_{L^2}^{p-\frac p\gamma}\|\theta\|_{L^p}^{\frac p\gamma}\le
    0.
    \end{equation}
Since $\gamma\in(0, 1)$ and $\|\theta(t)\|_{L^2}\le
\|\theta_0\|_{L^2}$, from \eqref{5.31p}, we obtain
\begin{equation}\label{5.32p}
    \frac{d}{dt}\|\theta(t)\|_{L^p}^p+2C\nu\|\theta_0\|_{L^2}^{p-\frac p\gamma}\|\theta\|_{L^p}^{\frac p\gamma}\le
    0.
    \end{equation} which, by integration, give
    \begin{eqnarray}\|\theta(t)\|_{L^p}\le
    \|\theta_0\|_{L^p}\Big(1+\frac{1-\gamma}\gamma\frac{2\nu
    C\|\theta_0\|_{L^p}^{\frac p\gamma-p}}{\|\theta_0\|_{L^2}^{\frac
    p\gamma-p}}t\Big)^{-\frac\gamma{p(1-\gamma)}}.\label{lpp}
    \end{eqnarray}

Finally, we need to estimate $\|\nabla\theta\|_{L^2}$.

Multiplying both sides of (\ref{1.1}) by $\Lambda^2\theta(t)$ and
integrating in $\mathbb{R}^N$, we get
\begin{equation}\label{5.11}
 \frac{1}{2}\frac{d}{dt}\int_{\mathbb{R}^N}\|\Lambda\theta\|_{L^2}^2dx+\nu\int_{\mathbb{R}^N}\|\Lambda^{\frac{\alpha}{2}+1}\theta(t)\|^2dx
 =-\int_{\mathbb{R}^N}div(u\theta)\Lambda^2\theta dx.
\end{equation}
The right-hand side of (\ref{5.11}) can be estimated by
\begin{eqnarray}\label{5.12}
  |\int_{\mathbb{R}^N}div(u\theta)\Lambda^2\theta
  dx|&=
  &|\int_{\mathbb{R}^N}(\xi_1\widehat{u_1\theta}(\xi)+\xi_1\widehat{u_2\theta}(\xi)+\cdots
   +\xi_1\widehat{u_N\theta}(\xi))
  |\xi|^2\hat{\theta}(\xi)d\xi|\nonumber\\
  &\leq&\sum\limits_{i=1}^N\int_{\mathbb{R}^N}|\xi|^{2-\frac{\alpha}{2}}|\widehat{\theta
  u_i}(\xi)||\xi|^{\frac{\alpha}{2}+1}|\hat{\theta}(\xi)|d\xi \nonumber\\
  &\leq&\sum\limits_{i=1}^N\|\Lambda^{2-\frac{\alpha}{2}}(\theta
  u_i)\|_2\|\Lambda^{\frac{\alpha}{2}+1}\theta\|_2 \nonumber\\
  &\leq&\frac{\nu}{4}\|\Lambda^{\frac{\alpha}{2}+1}\theta\|_2^2+\frac{2}{\nu}\sum\limits_{i=1}^N
  \|\Lambda^{2-\frac{\alpha}{2}}(\theta u_i)\|_2^2,
\end{eqnarray}
where we have used the Plancherel and H\"{o}lder inequality.

By the fractional calculus inequality \eqref{lam1} with $r=2$ and
Lemma \ref{plam}, we have
\begin{eqnarray}\label{5.14}
  \|\Lambda^{2-\frac{\alpha}{2}}(\theta u_i)\|_2 &\leq& C(\|u_i\|_q\|\Lambda^{2-\frac{\alpha}{2}}\theta\|_{p}
  +\|\theta\|_q\|\Lambda^{2-\frac{\alpha}{2}}u_i\|_{p})\nonumber \\
&\leq& C\|\theta\|_q\|\Lambda^{2-\frac{\alpha}{2}}\theta\|_p
\end{eqnarray}
for $i=1,2,\ldots, N$ and $\frac 1p+\frac 1q=\frac 12$. By the
maximum principle $\|\theta\|_{L^q}\leq C\|\theta_0\|_{L^q}$, we
have
\begin{equation}\label{5.15}
 \|\Lambda^{2-\frac{\alpha}{2}}(\theta u_i)\|_2\leq
 C(\theta_0)\|\Lambda^{2-\frac{\alpha}{2}}\theta\|_p
\end{equation}
for $i=1,2,\ldots, N$.

Using Lemma \ref{le2.3}, we get
\begin{equation}\label{5.16}
  \|\Lambda^{{2-\frac{\alpha}{2}}}\theta \|_p\leq C(\theta_0)\|\Lambda^{{2-\frac{\alpha}{2}}+\delta}\theta
  \|_2,\quad (i=1,2,\ldots, N),
\end{equation}
where $\frac{1}{p}=\frac{1}{2}-\frac{\delta}{N}$ and $0<\delta<N$.

Combining (\ref{5.11}), (\ref{5.12}), (\ref{5.14}), (\ref{5.15}) and
(\ref{5.16}), one get
\begin{equation}\label{5.17}
  \frac{1}{2}\frac{d}{dt}\|\Lambda\theta\|_{L^2}^2+\frac{3}{4}\nu\|\Lambda^{\frac{\alpha}{2}}\theta\|
  _{L^2}^2\leq
  C(\theta_0)\|\Lambda^{2-\frac{\alpha}{2}+\delta}\theta\|_{L^2}.
\end{equation}
For the right-hand of (\ref{5.17}), we have
\begin{eqnarray}\label{5.18}
 &&\|\Lambda^{2-\frac{\alpha}{2}+\delta}\theta(t)\|_{L^2}^2 \nonumber\\
 &=& \int_{B(t)}|\xi|^{4-\alpha
+2\delta}|\hat{\theta}(t)|^2d\xi+\int_{B(t)^c}|\xi|^{4-\alpha+2\delta}|\hat{\theta}(t)|^2
d\xi \nonumber\\
&\leq&
M^{4-\alpha+2\delta}(t)\|\theta(t)\|_{L^2}^2+\int_{B(t)^c}\frac{|\xi|^{4-\alpha+2\delta}}{|\xi|
  ^{2(\frac{\alpha}{2}+1)}}||\xi|^{\frac{\alpha}{2}+1}\hat{\theta}(t)|^2d\xi\nonumber\\
  &\leq& M^{4-\alpha+2\delta}(t)\|\theta(t)\|_{L^2}^2+\int_{B(t)^c}|\xi|^{-2(\frac{\alpha}{2}+1)+(4-\alpha+2\delta)}
 ||\xi|^{\frac{\alpha}{2}+1}\widehat{\theta(t)}|^2d\xi\nonumber \\
&\leq&
M^{4-\alpha+2\delta}(t)\|\theta(t)\|_{L^2}^2+M^{2(1-\alpha+\delta)}(t)
\|\Lambda^{\frac{\alpha}{2}+1}\theta(t)\|_{L^2}^2.
\end{eqnarray}
Because $\alpha>1$, we can choose $M$ large enough such that
$M^{2(1-\alpha+\delta)}(t)< \frac{\nu}{4C(\theta_0)}$. It follows
from \eqref{5.18} that
\begin{equation}\label{5.19}
  C(\theta_0)\|\Lambda^{2-\frac{\alpha}{2}+\delta}\theta(t)\|_{L^2}^2dx\leq
  C(\theta_0)M^{4-\alpha+2\delta}\|\theta(t)\|_{L^2}^2+\frac{\nu}{4}\|\Lambda^{\frac{\alpha}{2}+1}\theta\|_{L^2}^2.
\end{equation}
Putting \eqref{5.19} into \eqref{5.17}, we have
\begin{equation}
\frac{d}{dt}\|\Lambda\theta\|_{L^2}^2+\nu\|\Lambda^{1+\frac{\alpha}{2}}\theta\|_{L^2}^2\leq
C\|\theta(t)\|_{L^2}^2. \label{5.20}\end{equation} Moreover, we have
\begin{eqnarray}
 \|\Lambda^{\frac{\alpha}{2}+1}\theta\|_{L^2}^2&\geq& \int_{B(t)^c}|\xi|^{2(\frac{\alpha}{2}+1)}|\hat{\theta}|^2d\xi\nonumber\\
  &\geq&M^\alpha(t)\int_{B(t)^c}|\xi|^2|\hat{\theta}|^2d\xi \nonumber\\
 &=&M^\alpha(t)\|\Lambda\theta\|_{L^2}^2-M^\alpha(t)\int_{B(t)}|\xi|^2|\hat{\theta}|^2d\xi,\nonumber
\end{eqnarray}
which yields to
\begin{equation}\label{5.21}
  \|\Lambda^{\frac{\alpha}{2}+1}\theta\|_{L^2}^2\geq
  M^\alpha(t)\|\Lambda
  \theta\|_{L^2}^2-M^{\alpha+2}(t)\|\theta(t)\|_{L^2}^2.
\end{equation}
Combining \eqref{5.20} and \eqref{5.21}, we have \begin{equation}
\frac{d}{dt}\|\Lambda\theta\|_{L^2}^2+\nu
M^\alpha(t)\|\Lambda\theta\|_{L^2}^2\leq
C\|\theta\|_{L^2}^2+CM^{\alpha+2}(t)\|\theta\|_{L^2}^2,\quad \forall
M.\label{la1}
\end{equation}
It's obvious that we need to the obtained estimate
$\|\theta(t)\|_{L^2}^2$.

Putting (\ref{5.34}) into (\ref{la1}) and letting $M(t)=M$ be a
constant large enough such that $M^{\alpha+2}>1$, we have
\begin{equation}\label{5.22}
    \frac{d}{dt}\|\Lambda\theta\|_{L^2}^2+\nu M^\alpha\|\Lambda\theta\|_{L^2}^2\leq
    CM^{\alpha+2}(1+t)^{-\frac{1}{\beta}}.
\end{equation}
Then, by multiplying $e^{M^\alpha t}$ on (\ref{5.22}) and
integrating with respect to $t$, we obtain
\begin{eqnarray}\label{5.23}
 \|\Lambda\theta\|_{L^2}^2&\leq&e^{-M^\alpha t}\|\Lambda\theta_0\|_{L^2}^2+CM^{\alpha+2}\int_0^te^
 {-M^\alpha(t-s)}(1+s)^{-\frac1\beta}ds \nonumber\\
 &\leq& e^{-M^\alpha t}\|\Lambda
 \theta_0\|_{L^2}^2+CM^{\alpha+2}(1+t)^{-\frac1\beta},
\end{eqnarray}
where we have used the estimate
$$\int_0^t e^{-M^\alpha(t-s)}(1+s)^{-\frac{1}{\alpha}}ds\leq
C(1+t)^{-\frac{1}{\alpha}}, t>0.$$ Thanks to the fact
$\|\nabla\theta\|_{L^2}=\|\Lambda\theta\|_{L^2}$, it follows from
\eqref{5.23} that \begin{eqnarray}\|\nabla\theta\|_{L^2}\leq
(1+t)^{-\frac12(\frac{N+2-2\alpha}{\alpha}-\epsilon)}.\label{decayp}\end{eqnarray}
The estimates \eqref{5.34}, \eqref{lpp} and \eqref{decayp} give the
desire decay estimates. This completes the proof of Theorem
\ref{thm1.7}.
%%%%%%%%%%%%%%%%%%%%%%%%%%%%%%%%%%%%%%%%%%%%%%%%%%%%%%%%%%%%%%%%%%%%%%%%%%%%%
%%%%%%%%%%%%%%%%%%%%%%%%%%%%%%%%%%%%%%%%%%%%%%%%%%%%%%%%%%%%%%%%%%%%%%%%%%%%%%%%%%%%%%%%%%%%%%%
%%%%%%%%%%%%%%%%%%%%%%%%%%%%%%%%%%%%%%%%%%%%%%%%%%%%%%%%%%%%%%%%%%%%%%%%%%%%%%%%%%%%%%%%%%%%
\paragraph{Acknowledgments.}
The research of Prof. S. Wang was supported by National Basic
Research Program of China (973 Program, 2011CB808002), the NSFC
(11071009) and PHR-IHLB (200906103).

\end{document}